\RequirePackage{fix-cm}
\documentclass[smallcondensed,nospthms]{svjour3}     
\smartqed  
\usepackage[T1]{fontenc}
\usepackage[utf8]{inputenc}

\usepackage{graphicx}
\usepackage{booktabs}
\usepackage[dvipsnames]{xcolor}
\usepackage{hyperref}

\usepackage{amsmath}
\usepackage{amsfonts}
\usepackage{amsthm}
\usepackage{mathtools}
\usepackage{braket} 
\usepackage{empheq} 
\usepackage{stackrel}

\theoremstyle{plain} 
\newtheorem{theorem}{Theorem}[section]
\newtheorem{lemma}[theorem]{Lemma}

\theoremstyle{definition}
\newtheorem{definition}[theorem]{Definition}
\newtheorem{assumption}[theorem]{Assumption} 
\theoremstyle{remark}
\newtheorem{remark}[theorem]{Remark}

\numberwithin{equation}{section}

\DeclareMathOperator{\dive}{div}

\begin{document}

\title{Analysis of the SORAS domain decomposition preconditioner for non-self-adjoint or indefinite problems
}
\titlerunning{SORAS preconditioner analysis for non-SPD problems}        

\author{Marcella Bonazzoli \and Xavier Claeys \and Fr\'ed\'eric Nataf \and Pierre-Henri Tournier}
\authorrunning{M. Bonazzoli, X. Claeys, F. Nataf, P.-H. Tournier} 

\institute{M. Bonazzoli \at
              Inria, Centre de Math\'ematiques Appliqu\'ees, CNRS, \'Ecole Polytechnique, Institut Polytechnique de Paris, 91128 Palaiseau, France.\\
              \email{marcella.bonazzoli@inria.fr}, corresponding author           
           \and
           X. Claeys \and  F. Nataf \and P.-H. Tournier \at
              Sorbonne Universit\'e, CNRS, Universit\'e de Paris, Inria, Laboratoire Jacques-Louis Lions, F-75005 Paris, France. 
              \email{xavier.claeys@sorbonne-universite.fr, frederic.nataf@sorbonne-universite.fr, tournier@ljll.math.upmc.fr}
}

\date{Received: date / Accepted: date}

\maketitle

\begin{abstract}
We analyze the convergence of the one-level overlapping domain decomposition preconditioner SORAS (Symmetrized Optimized Restricted Additive Schwarz) applied to a generic linear system whose matrix is not necessarily symmetric/self-adjoint nor positive definite. By generalizing the theory for the Helmholtz equation developed in [I.G. Graham, E.A. Spence, and J. Zou, {\em SIAM J.~Numer.~Anal.}, 2020], we identify a list of assumptions and estimates that are sufficient to obtain an upper bound on the norm of the preconditioned matrix, and a lower bound on the distance of its field of values from the origin. We stress that our theory is general in the sense that it is not specific to one particular boundary value problem. Moreover, it does not rely on a coarse mesh whose elements are sufficiently small. As an illustration of this framework, we prove new estimates for overlapping domain decomposition methods with Robin-type transmission conditions for the heterogeneous reaction-convection-diffusion equation {(to prove the stability assumption for this equation we consider the case of a coercive bilinear form, which is non-symmetric, though)}. 
\keywords{Non-self-adjoint problems \and indefinite problems \and domain decomposition \and preconditioners \and field of values \and reaction-convection-diffusion equation}
\subclass{65N55 \and 65F08 \and 65F10 \and 76R99}
\end{abstract}

\section{Introduction}



The discretization of several partial differential equations relevant in applications, such as the Helmholtz equation, the time-harmonic Maxwell equations or the reaction-convection-diffusion equation, yields linear systems whose matrices are not symmetric/self-adjoint or indefinite. 
The rigorous analysis of the convergence of preconditioned iterative methods for such problems is harder than for symmetric positive definite (SPD) problems. 
Indeed, in the SPD case, Hilbert space theorems such as the Fictitious Space lemma (see e.g.~\cite{Nepomnyaschikh:1991:MTT,Griebel:1995:ATA}) yield a powerful general framework of spectral analysis for domain decomposition preconditioners. In addition, in the non-SPD case the conjugate gradient method cannot be used, and the analysis of the spectrum of the preconditioned matrix is not sufficient for iterative methods such as GMRES suited for non-self-adjoint matrices. In fact, as stated in~\cite{Greenbaum:1996:ANI}, ``any nonincreasing convergence curve can be obtained with GMRES applied to a matrix having any desired eigenvalues''. In the literature, GMRES convergence estimates are based for instance on the field of values~\cite{Elman:1982:IML,Eisenstat:1983:VIM,Beckermann:2005:SRE} or on the pseudo-spectrum (see~\cite{Trefethen:2005:SSB} and references therein) of the preconditioned operator. For example, field of values bounds were derived for overlapping domain decomposition preconditioners {for non-symmetric parabolic problems which are small perturbations of SPD operators \cite{Chan:1996:ACT}, and later} for the high-frequency Helmholtz~\cite{GrSpZo:Helm:2017,GrSpZo:impedance,GoGrSp:heterHelm} and time-harmonic Maxwell~\cite{BoDoGrSpTo:Maxwell} equations. 

Here, by generalizing the work of~\cite{GrSpZo:impedance}, we analyze for generic problems the convergence of the preconditioned GMRES method in its weighted version~\cite{Essai:1998:wGMRES}. We identify a list of assumptions and estimates that are sufficient to obtain an upper bound on the norm of the preconditioned matrix, and a lower bound on the distance of its field of values from the origin. This analysis applies to a class of one-level overlapping Schwarz domain decomposition preconditioners, with Robin-type or more general absorbing transmission conditions on the interfaces between subdomains. This type of preconditioners with the basic Robin-type transmission conditions was first introduced in (\cite{KiSa:OBDD:2007}, 2007) for the Helmholtz equation and called OBDD-H (\emph{Overlapping Balancing Domain Decomposition for Helmholtz}). It was later studied in (\cite{HaJoNa:soras:2015}, 2015) for generic symmetric positive definite problems and viewed as a symmetric variant of the ORAS preconditioner (\cite{StGaTh:ORAS:2007}, 2007), hence called SORAS (\emph{Symmetrized Optimized Restricted Additive Schwarz}). 
Note that in \cite{KiSa:OBDD:2007} several one-level and two-level versions, with a coarse space based on plane waves, were tested numerically, and more than ten years later the one-level OBDD-H version was rigorously analyzed in \cite{GrSpZo:impedance}, for the Helmholtz equation.    
In \cite{HaJoNa:soras:2015} a two-level version, with a spectral coarse space, was rigorously analyzed for generic SPD problems. 
The present article gives, to the best of our knowledge, the first rigorous analysis of such one-level preconditioners for \emph{generic} non-SPD problems.

Furthermore, we apply our general framework to the case of convection-diffusion equations to obtain, for the first time, convergence bounds for one-level overlapping Schwarz domain decomposition preconditioners with Robin-type transmission conditions. 
For these equations, the two-level overlapping case, but with standard Dirichlet transmission conditions, was analyzed in~\cite{Cai:1991:ASA,Cai:1992:DAI}, where a coarse space is built from a coarse mesh whose elements are sufficiently small. As for the one-level non-overlapping case, it was studied with Robin or more general transmission conditions in e.g.~\cite{Lube:2000:AND,Nataf:1995:FCD}, see also \cite{JaNaRo:2001:OO2} for some numerical results. Apart from Schwarz methods, the Neumann–Neumann algorithm~\cite{Bourgat:1989:VFA}, which belongs to the substructuring family of domain decomposition methods, was generalized to convection-diffusion equations in \cite{Achdou:2000:DDP}, and a coarse space not based on a coarse mesh was proposed in~\cite{Alart:2000:MSA} although without convergence analysis. 

The paper is structured as follows. 
In section~\ref{sec:setting} we first describe in detail the considered class of domain decomposition preconditioners and introduce notation for the global and local inner products and norms. In section~\ref{sec:genThm} we state and prove the main theorem, which provides a general and practical tool for the rigorous convergence analysis of the preconditioner. This framework is applied in section~\ref{sec:RCDeq} to the case of the heterogeneous reaction-convection-diffusion equation. After specifying the global and local bilinear forms, inner products and norms and the discretization, we prove estimates for the assumptions of the theorem for this equation, without making any a priori assumption on the regime of the physical coefficients nor of the numerical parameters; {to prove the stability assumption for this equation we consider the case of a coercive bilinear form (which is non-symmetric, though)}. Finally, we discuss for a particular regime the resulting lower bound on the field of values, {and we test numerically the performance of the preconditioner.}

\section{Setting}
\label{sec:setting}

Let $A$ denote the $n \times n$ (potentially complex-valued) matrix arising from the discretization of the problem to be solved, posed in an open domain $\Omega \subset \mathbb{R}^d$. The matrix $A$ is not necessarily positive definite nor self-adjoint. This means that here we do \emph{not} necessarily require $A^*=A$, where $A^* \coloneqq \overline{A^{T}}$; note that ``self-adjoint matrix'' is a synonym for ``Hermitian matrix''. In particular, if $A$ is real-valued this means that here it does not need to be symmetric. 

The definition of the preconditioner is based on a set of overlapping open subdomains $\Omega_j, j=1,\dots,N$, such that $\Omega = \cup_{j=1}^N \Omega_j$ and each $\overline{\Omega_j}$ is a union of elements of the mesh $\mathcal{T}^h$ of $\Omega$. 
Then we consider the set $\mathcal{N}$ of the unknowns on the whole domain, so $\#\mathcal{N} = n$, and its decomposition $\mathcal{N} = \bigcup_{j=1}^N \mathcal{N}_j$ into the non-disjoint subsets corresponding to the different overlapping subdomains $\Omega_j$, with $\#\mathcal{N}_j = n_j$. Then one builds the following matrices (see e.g. \cite[§1.3]{DoJoNa:bookDDM}):
\begin{itemize}
\item
the \emph{restriction} matrices $R_j$ from $\Omega$ to the subdomain $\Omega_j$: they are $n_j \times n$ Boolean matrices whose $(i,i')$ entry equals $1$ if the $i$-th unknown in $\mathcal{N}_j$ is the $i'$-th one in $\mathcal{N}$ and vanishes otherwise; 
\item
the \emph{extension} by zero matrices from the subdomain $\Omega_j$ to $\Omega$, which are $n \times n_j$ Boolean matrices given by $R^T_j$;
\item
the \emph{partition of unity} matrices $D_j$, which are $n_j \times n_j$ diagonal matrices with real non-negative entries such that $\sum_{j=1}^N R_j^T D_j R_j = I$. They can be seen as matrices that properly weight the unknowns belonging to the overlap between subdomains;
\item
the \emph{local matrices} $B_j$, of size $n_j \times n_j$, arising from the discretization of subproblems posed in $\Omega_j$, with for instance Robin-type or absorbing\footnote{Absorbing boundary conditions are approximations of transparent boundary conditions. Basic absorbing boundary conditions are Robin-type boundary conditions, which consist in a  weighted combination of Neumann-type and Dirichlet-type boundary conditions. Their precise definition depends on the specific problem. For instance, for Maxwell equations impedance boundary conditions are Robin-type absorbing boundary conditions.} transmission conditions on the interfaces $\partial\Omega_j \setminus \partial\Omega$. 
\end{itemize}
Finally, the one-level Symmetrized Optimized Restricted Additive Schwarz (SORAS) preconditioner is defined as
\begin{equation}
\label{eq:SORAS}
M^{-1} \coloneqq \sum_{j=1}^N R_j^T D_j B_j^{-1} D_j R_j. 
\end{equation}
Note that here the preconditioner is not self-adjoint when $B_j$ is not self-adjoint, even if we maintain the SORAS name, where S stands for `Symmetrized'. 
In fact, this denomination was introduced in \cite{HaJoNa:soras:2015} for SPD problems, since in that case the SORAS preconditioner is a symmetric variant of the ORAS preconditioner $\sum_{j=1}^N R_j^T D_j B_j^{-1} R_j$. Thus, the adjective `Symmetrized' stands for the presence of the rightmost partition of unity $D_j$. We recall that the adjective `Restricted' indicates the presence of the leftmost partition of unity $D_j$. The adjective `Optimized' refers to the choice of transmission conditions other than standard Dirichlet conditions in the local matrices $B_j$, which can be better suited to the problem at hand and accelerate the convergence of the method.

The weighted GMRES method \cite{Essai:1998:wGMRES} differs from the standard one in the norm used for the residual minimization, which is not the standard Hermitian norm but a more general weighted norm. For vectors of degrees of freedom $\mathbf{V}, \mathbf{W} \in \mathbb{C}^n$, using the notation $(\mathbf{V}, \mathbf{W}) \coloneqq \mathbf{W}^*\mathbf{V}$ to indicate the Hermitian inner product, given a $n \times n$ self-adjoint positive definite matrix $F_\Omega$, we consider the weighted norm
\[
\Vert \mathbf{V} \rVert_{\Omega} \coloneqq (\mathbf{V}, \mathbf{V})_{F_\Omega}^{1/2}, \quad \text{where } (\mathbf{V}, \mathbf{W})_{F_\Omega} \coloneqq  (F_\Omega \mathbf{V}, \mathbf{W}) = \mathbf{W}^* F_\Omega \mathbf{V}.   
\]
Locally, on the subdomain $\Omega_j$, we consider a weighted norm represented by a $n_j \times n_j$ self-adjoint positive definite matrix $F_{\Omega_j}$: 
for vectors of degrees of freedom $\mathbf{V}^j, \mathbf{W}^j \in \mathbb{C}^{n_j}$ local to $\Omega_j$, 
we define
\[
\Vert \mathbf{V}^j \rVert_{\Omega_j} \coloneqq (\mathbf{V}^j, \mathbf{V}^j)_{F_{\Omega_j} }^{1/2}, \quad \text{where } (\mathbf{V}^j, \mathbf{W}^j)_{F_{\Omega_j}} \coloneqq (F_{\Omega_j} \mathbf{V}^j, \mathbf{W}^j)  = (\mathbf{W}^j)^* F_{\Omega_j} \mathbf{V}^j.
\]
Typically $F_{\Omega_j}$ is a Neumann-type matrix on $\Omega_j$, that is, coming from an inner product at the continuous level with no boundary integral.

\section{General theory}
\label{sec:genThm}

In order to apply Elman-type estimates for the convergence of weighted GMRES \cite{Essai:1998:wGMRES}, such as \cite[Theorem 5.1]{GrSpZo:Helm:2017} or its improvement \cite[Theorem 5.3]{BoDoGrSpTo:Maxwell}, we need to prove an upper bound on the weighted norm of the preconditioned matrix, and a lower bound on the distance of its weighted field of values from the origin. 
Recall that the \emph{field of values (or numerical range)} of a matrix $C$ \emph{with respect to} the inner product induced by a matrix $F$ is the set defined as
\[
W_F (C) = \set{ (\mathbf{V}, C \mathbf{V})_F | \mathbf{V} \in \mathbb{C}^n, \lVert \mathbf{V} \rVert_F=1 }.
\]
(Note that the convergence estimate for GMRES based on the field of values can be used only when this latter does not contain $0$.) 

The following theorem, which generalizes the theory for the Helmholtz equation developed in \cite{GrSpZo:impedance}, identifies assumptions that are sufficient to obtain the two bounds. 
In particular, the proof was inspired by the one of \cite[Theorem 3.11]{GrSpZo:impedance} and by the analysis in subsection \cite[§3.2]{GrSpZo:impedance}. 

We will need the notation for the commutator $[P,Q] \coloneqq PQ-QP$. 

\begin{theorem}
\label{thm:main}

For $j=1,\dots,N$, assume that for all global vectors of degrees of freedom $\mathbf{V} \in \mathbb{C}^n$ and local vectors of degrees of freedom $\mathbf{W}^j \in \mathbb{C}^{n_j}$ in $\Omega_j$
\begin{equation}
\label{eq:DAB}
(D_j R_j A \mathbf{V},\mathbf{W}^j)= (D_j B_j R_j\mathbf{V},\mathbf{W}^j).
\end{equation}
Suppose that there exists $\Lambda_0>0$ such that for all local vectors of degrees of freedom $\mathbf{W}^j \in \mathbb{C}^{n_j}$ in $\Omega_j$, $j=1,\dots,N$, we have 
\begin{equation}
\label{eq:converseStSp}
\biggl \lVert  \sum_{j=1}^N R_j^T  \mathbf{W}^j \biggr \rVert_\Omega^2 \le 
\Lambda_0 \sum_{j=1}^N \lVert \mathbf{W}^j\rVert_{\Omega_j}^2, 
\end{equation}
and $\Lambda_1>0$ such that for all global vectors of degrees of freedom $\mathbf{V} \in \mathbb{C}^n$
\begin{equation}
\label{eq:finOvr}
 \sum_{j=1}^N \lVert R_j \mathbf{V}\rVert_{\Omega_j}^2 
 \le \Lambda_1 \Vert \mathbf{V} \rVert_\Omega^2.
\end{equation}
For $j=1,\dots,N$, suppose also that there exist $C_{D,j}, C_{DB,j} > 0$ such that for all local vectors of degrees of freedom $\mathbf{W}^j, \mathbf{V}^j \in \mathbb{C}^{n_j}$ in $\Omega_j$
\begin{align}
&\lVert D_j \mathbf{W}^j \rVert_{\Omega_j} \le C_{D,j} \lVert \mathbf{W}^j  \rVert_{\Omega_j}, \label{eq:sim3-16} \\
& \lvert ([D_j, B_j] \mathbf{V}^j, \mathbf{W}^j) \rvert \le C_{DB,j} \lVert \mathbf{V}^j \rVert_{\Omega_j}  \lVert \mathbf{W}^j \rVert_{\Omega_j}, \label{eq:cDB}
\end{align}
and that $B_j$ satisfies the following inf-sup condition: there exists $C_{\textup{stab},j}>0$ such that for all local vectors of degrees of freedom $\mathbf{U}^j \in \mathbb{C}^{n_j}$
\begin{equation}
\label{eq:stability}
\lVert \mathbf{U}^j \rVert_{\Omega_j} \le 
C_{\textup{stab},j} \max_{\mathbf{W}^j \in \mathbb{C}^{n_j}\setminus\{0\}} \left (\frac{\lvert (B_j \mathbf{U}^j, \mathbf{W}^j) \rvert}{\lVert \mathbf{W}^j \rVert_{\Omega_j}} \right). 
\end{equation} 

Then, we obtain the following upper bound on the norm of the preconditioned matrix: 
\begin{empheq}[box=\fbox]{equation}
\label{eq:Ubound}
\max_{\mathbf{V} \in \mathbb{C}^n} \; \frac{\Vert M^{-1} A \mathbf{V} \rVert_\Omega}{\Vert \mathbf{V} \rVert_\Omega} 
\le \sqrt{\Lambda_0 \Lambda_1} \max_{j=1,\dots,N} \{ C_{D,j} (C_{\textup{stab},j} C_{DB,j} + C_{D,j})\}.
\end{empheq}

If in addition, for $j=1,\dots,N$, for all global vectors of degrees of freedom $\mathbf{V} \in \mathbb{C}^n$ and local vectors of degrees of freedom $\mathbf{W}^j \in \mathbb{C}^{n_j}$ in $\Omega_j$
\begin{equation}
\label{eq:FkD}
(D_j R_j F_\Omega \mathbf{V}, \mathbf{W}^j) = (D_j F_{\Omega_j} R_j \mathbf{V}, \mathbf{W}^j), 
\end{equation}
and there exists $C_{DF,j} > 0$ such that for all local vectors of degrees of freedom $\mathbf{V}^j, \mathbf{W}^j \in \mathbb{C}^{n_j}$ in $\Omega_j$
\begin{equation}
\label{eq:cDF}
\lvert ([D_j, F_{\Omega_j}] \mathbf{V}^j, \mathbf{W}^j) \rvert \le C_{DF,j} \lVert \mathbf{V}^j \rVert_{\Omega_j}  \lVert \mathbf{W}^j \rVert_{\Omega_j},
\end{equation}
then 
we obtain the following lower bound on the distance of the field of values of the preconditioned matrix from the origin:
\begin{empheq}[box=\fbox]{multline}
\label{eq:Lbound}
\min_{\mathbf{V} \in \mathbb{C}^n} \frac{\lvert (F_\Omega \mathbf{V}, M^{-1}A \mathbf{V}) \rvert}{ \lVert \mathbf{V} \rVert_\Omega^2} \ge 
 \frac{1}{\Lambda_0} 
-  \Lambda_1 \max_{j=1,\dots,N} \{ C_{D,j} C_{\textup{stab},j}C_{DB,j} \} \\
- \Lambda_1 \max_{j=1,\dots,N}  \{ C_{DF,j} (C_{\textup{stab},j}C_{DB,j} +C_{D,j}) \}.
\end{empheq}

\end{theorem}

\begin{remark}
We will comment on assumptions \eqref{eq:DAB}, \eqref{eq:converseStSp}, \eqref{eq:finOvr}, \eqref{eq:FkD} in subsection \ref{subsec:commentsHp}. 
Note that in finite dimension, the constants in assumptions \eqref{eq:sim3-16}, \eqref{eq:cDB}, \eqref{eq:stability}, \eqref{eq:cDF} are finite, and in the statement of the theorem we actually mean that we are able to estimate these constants. 
\end{remark}


\begin{proof}
To obtain both bounds an important quantity is 
\[
\lVert (B_j^{-1} D_j R_j A - D_j R_j)\mathbf{V} \rVert_{\Omega_j}.
\]
For its estimate, for any vector of degrees of freedom $\mathbf{W}^j \in \mathbb{C}^{n_j}$ local to $\Omega_j$, write 
\begin{equation*}
\begin{split}
(B_j (B_j^{-1} D_j R_j A - D_j R_j)\mathbf{V}, \mathbf{W}^j) &= (D_j R_j A \mathbf{V}, \mathbf{W}^j) - (B_j D_j R_j\mathbf{V}, \mathbf{W}^j) \\
&\stackrel{\eqref{eq:DAB}}{=} (D_j B_j R_j \mathbf{V}, \mathbf{W}^j) - (B_j D_j R_j\mathbf{V}, \mathbf{W}^j) \\
&=  ([D_j, B_j] R_j \mathbf{V}, \mathbf{W}^j),
\end{split}
\end{equation*} 
where assumption \eqref{eq:DAB} was used. Thus we have found that $(B_j^{-1} D_j R_j A - D_j R_j)\mathbf{V}$ is the {solution to a local problem} with a right-hand side involving the commutator between the partition of unity and the local matrix. 
So by the stability bound \eqref{eq:stability}, we have:
\[
\lVert (B_j^{-1} D_j R_j A - D_j R_j)\mathbf{V} \rVert_{\Omega_j} \le 
C_{\textup{stab},j} \max_{\mathbf{W}^j \in \mathbb{C}^{n_j}\setminus\{0\}} \left (\frac{\lvert([D_j, B_j] R_j \mathbf{V}, \mathbf{W}^j)\rvert}{\lVert \mathbf{W}^j \rVert_{\Omega_j}} \right).
\]
Moreover by assumption \eqref{eq:cDB}
\[
\lvert ([D_j, B_j] R_j \mathbf{V}, \mathbf{W}^j) \rvert \le C_{DB,j} \lVert R_j\mathbf{V} \rVert_{\Omega_j}  \lVert \mathbf{W}^j \rVert_{\Omega_j} \; \forall \, \mathbf{W}^j.
\]
Therefore
\begin{equation}
\label{eq:estSigma}
\lVert (B_j^{-1} D_j R_j A - D_j R_j)\mathbf{V} \rVert_{\Omega_j} \le C_{\textup{stab},j} C_{DB,j} \lVert R_j \mathbf{V} \rVert_{\Omega_j}.
\end{equation}
Together with \eqref{eq:estSigma}, a direct consequence of \eqref{eq:estSigma} itself and assumption \eqref{eq:sim3-16} will be also used repeatedly: 
\begin{equation}
\label{eq:sim3-34bis}
\lVert B_j^{-1} D_j R_j A \mathbf{V}  \rVert_{\Omega_j}  \le (C_{\textup{stab},j} C_{DB,j} + C_{D,j}) \lVert R_j \mathbf{V}  \rVert_{\Omega_j}.
\end{equation}

Now, it is easy to obtain the upper bound \eqref{eq:Ubound}: for $\mathbf{V} \in \mathbb{C}^n$ we have 
\[
\begin{split}
\biggl \lVert \sum_{j=1}^N R_j^T D_j B_j^{-1} D_j R_j A \mathbf{V} & \biggr \rVert_\Omega^2 
\stackrel{\eqref{eq:converseStSp}}{\le} \Lambda_0  \sum_{j=1}^N \lVert D_j B_j^{-1} D_j R_j A \mathbf{V}\rVert_{\Omega_j}^2 \\
&\stackrel{\eqref{eq:sim3-16}}{\le} \Lambda_0 \sum_{j=1}^N C_{D,j}^2 \lVert B_j^{-1} D_j R_j A \mathbf{V}\rVert_{\Omega_j}^2 \\
&\stackrel{\eqref{eq:sim3-34bis}}{\le} \Lambda_0 \sum_{j=1}^N C_{D,j}^2(C_{\textup{stab},j} C_{DB,j} + C_{D,j})^2 \lVert R_j \mathbf{V}  \rVert_{\Omega_j}^2 \\
&\stackrel{\eqref{eq:finOvr}}{\le} \Lambda_0 \Lambda_1 \max_{j=1,\dots,N}  \{ C_{D,j}^2  (C_{\textup{stab},j} C_{DB,j} + C_{D,j} )^2 \} \lVert \mathbf{V}  \rVert_{\Omega}^2,
\end{split}
\]
where we have indicated above each inequality sign which equation was used.  

\medskip
The derivation of the lower bound \eqref{eq:Lbound} is more involved. First of all write 
\[
\begin{split}
&(F_\Omega \mathbf{V}, \sum_{j=1}^N R_j^T D_j B_j^{-1} D_j R_j A \mathbf{V}) =
\sum_{j=1}^N (F_\Omega \mathbf{V}, R_j^T D_j B_j^{-1} D_j R_j A \mathbf{V}) \\
& = \sum_{j=1}^N (D_j R_j F_\Omega \mathbf{V}, B_j^{-1} D_j R_j A \mathbf{V}) 
\stackrel{\eqref{eq:FkD}}{=} \sum_{j=1}^N (D_j  F_{\Omega_j} R_j \mathbf{V}, B_j^{-1} D_j R_j A \mathbf{V}),
\end{split}
\]
where, beside applying assumption \eqref{eq:FkD}, we have used the fact that the partition of unity matrices $D_j$ are real-valued and diagonal, hence symmetric, and the restriction matrices $R_j$ satisfy $(\mathbf{V},R_j^T \mathbf{W}^j) = (R_j \mathbf{V}, \mathbf{W}^j)$. 
Now, we make appear the commutator between the partition of unity and the local inner product matrix, and also the quantity $(B_j^{-1} D_j R_j A - D_j R_j)\mathbf{V}$:
\[
\begin{split}
& (D_j  F_{\Omega_j} R_j \mathbf{V}, B_j^{-1} D_j R_j A \mathbf{V}) \\
& = (F_{\Omega_j} D_j R_j \mathbf{V}, B_j^{-1} D_j R_j A \mathbf{V}) +  ([D_j,F_{\Omega_j}] R_j \mathbf{V}, B_j^{-1} D_j R_j A \mathbf{V})\\
&= (F_{\Omega_j} D_j R_j \mathbf{V}, D_j R_j \mathbf{V}) + (F_{\Omega_j} D_j R_j \mathbf{V}, (B_j^{-1} D_j R_j A -D_j R_j)\mathbf{V}) \\
& \qquad +  ([D_j,F_{\Omega_j}] R_j \mathbf{V}, B_j^{-1} D_j R_j A \mathbf{V}).
\end{split}
\]
Therefore
\begin{equation}
\label{eq:3terms}
\begin{split}
&\lvert (F_\Omega \mathbf{V}, M^{-1}A \mathbf{V}) \rvert \ge \\
&\sum_{j=1}^N \Vert D_j R_j \mathbf{V} \rVert_{\Omega_j}^2 
-  \sum_{j=1}^N \lvert(F_{\Omega_j} D_j R_j \mathbf{V}, (B_j^{-1} D_j R_j A -D_j R_j)\mathbf{V}) \rvert \\
&- \sum_{j=1}^N \lvert([D_j,F_{\Omega_j}] R_j \mathbf{V}, B_j^{-1} D_j R_j A \mathbf{V}) \rvert.
\end{split}
\end{equation}
For the first term in \eqref{eq:3terms} we use the partition of unity property $\sum_{j=1}^N R_j^T D_j R_j = I$ and assumption \eqref{eq:converseStSp} with $\mathbf{W}^j = D_j R_j \mathbf{V}$:
\[
\lVert \mathbf{V} \rVert_\Omega^2 = \biggl \lVert  \sum_{j=1}^N R_j^T (D_j R_j \mathbf{V}) \biggr \rVert_\Omega^2 \stackrel{\eqref{eq:converseStSp}}{\le} 
\Lambda_0 \sum_{j=1}^N \lVert D_j R_j \mathbf{V}\rVert_{\Omega_j}^2, 
\]
so
\[
\sum_{j=1}^N \lVert D_j R_j \mathbf{V}\rVert_{\Omega_j}^2 \ge \frac{1}{\Lambda_0} \lVert \mathbf{V} \rVert_\Omega^2. 
\]
For the second term in \eqref{eq:3terms}, we use first the Cauchy-Schwarz inequality: 
\[
\begin{split}
&\sum_{j=1}^N \lvert(F_{\Omega_j} D_j R_j \mathbf{V}, (B_j^{-1} D_j R_j A -D_j R_j)\mathbf{V}) \rvert \\
&\quad \le \sum_{j=1}^N \lVert D_j R_j \mathbf{V}\rVert_{\Omega_j} \lVert (B_j^{-1} D_j R_j A -D_j R_j)\mathbf{V} \rVert_{\Omega_j} \\
& \; \stackrel{\eqref{eq:sim3-16}, \eqref{eq:estSigma}}{\le} \sum_{j=1}^N C_{D,j}  C_{\textup{stab},j} C_{DB,j} \lVert R_j \mathbf{V} \rVert_{\Omega_j}^2 \\
& \quad \stackrel{\eqref{eq:finOvr}}{\le} \Lambda_1  \max_{j=1,\dots,N} \{ C_{D,j}C_{\textup{stab},j} C_{DB,j} \} \lVert \mathbf{V} \rVert_\Omega^2.
\end{split}
\]
Finally for the third term in \eqref{eq:3terms} we write
\[
\begin{split}
&\sum_{j=1}^N \vert ([D_j,F_{\Omega_j}] R_j \mathbf{V}, B_j^{-1} D_j R_j A \mathbf{V}) \rvert\\
& \quad \stackrel{\eqref{eq:cDF}}{\le} \sum_{j=1}^N C_{DF,j}  \lVert R_j \mathbf{V} \rVert_{\Omega_j}  \lVert B_j^{-1} D_j R_j A  \mathbf{V}\rVert_{\Omega_j} \\
& \quad \stackrel{\eqref{eq:sim3-34bis}}{\le} \sum_{j=1}^N C_{DF,j} (C_{\textup{stab},j} C_{DB,j} + C_{D,j}) \lVert R_j \mathbf{V} \rVert_{\Omega_j}^2 \\
& \quad \stackrel{\eqref{eq:finOvr}}{\le} \Lambda_1 \max_{j=1,\dots,N} \{ C_{DF,j} (C_{\textup{stab},j} C_{DB,j} +C_{D,j}) \} \lVert \mathbf{V} \rVert_\Omega^2.
\end{split}
\]
In conclusion, inserting these estimations in \eqref{eq:3terms} we obtain the lower bound \eqref{eq:Lbound}.

\end{proof}


\subsection{Comments on the assumptions of Theorem \ref{thm:main}}
\label{subsec:commentsHp}

Assumptions \eqref{eq:DAB} and \eqref{eq:FkD} {relate the global matrices with the local ones through the partition of unity and restriction matrices.} They may appear unconventional at first glance, but they are satisfied for quite natural choices of the local sesquilinear form and continuous norm on the subdomains. 
More precisely, if the $i$-th entry of the diagonal of $D_j$ is not zero, assumption \eqref{eq:DAB} requires that the $i$-th rows of $R_j A$ and $B_j R_j$ are equal; likewise assumption \eqref{eq:FkD} requires that the $i$-th rows of $R_j F_\Omega$ and $F_{\Omega_j} R_j$ are equal. 
First of all, note that typically the entries corresponding to $\partial \Omega_j \setminus \partial \Omega$ of the partition of unity $D_j$ are zero. Moreover, $B_j$ arises from the discretization of a local sesquilinear form that usually is like the global sesquilinear form yielding $A$ but with the integrals on $\Omega_j$ instead of $\Omega$ and with an additional boundary integral on $\partial \Omega_j \setminus \partial \Omega$. In this case assumption \eqref{eq:DAB} is satisfied. Likewise, assumption \eqref{eq:FkD} is satisfied if the local continuous norm yielding $F_{\Omega_j}$ is obtained from the global continuous norm yielding $F_\Omega$ just by replacing $\Omega$ with $\Omega_j$ in the integration domain. As an illustration, see the bilinear forms $a$, $a_j$ and the continuous norms $\lVert \cdot \rVert_{1,c}$, $\lVert \cdot \rVert_{1,c, \Omega_j}$ defined in §\ref{sec:RCDeq} for the reaction-convection-diffusion equation and the proof of Lemma~\ref{lem:hypglobloc}: {in this case the essential properties on the continuous level are those expressed in Remak~\ref{rem:globloc}.}

Assumptions \eqref{eq:converseStSp} and \eqref{eq:finOvr} are classical inequalities in the domain decomposition framework. 
Inequality \eqref{eq:converseStSp} is dubbed in \cite{GrSpZo:impedance} `a kind of converse to the stable splitting result', and it can be viewed as a continuity property of the reconstruction operator $\{ \mathbf{W}^j \}_{j=1}^N \mapsto  \sum_{j=1}^N R_j^T  \mathbf{W}^j$. 
In \cite[Lemma 3.6]{GrSpZo:impedance} the inequality is proved at the continuous level for the Helmholtz energy norm (see \cite[eq. (1.15)]{GrSpZo:impedance}) with  
\begin{equation*}
\Lambda_0 = \max_{j=1,\dots,N} \, \# \Lambda(j), \quad \text{where } \Lambda(j) \coloneqq \set{i | \Omega_j \cap \Omega_{i} \ne \emptyset},  
\end{equation*}
in other words, $\Lambda_0$ is the maximum number of neighboring  subdomains.   
Note that the proof in \cite[Lemma 3.6]{GrSpZo:impedance} (essentially consisting in the one in \cite[eq.~(4.8)]{GrSpZo:Helm:2017}) is more generally valid, for instance whenever the local continuous norm can be obtained from the global continuous norm just by replacing $\Omega$ with $\Omega_j$ in the integration domain, as before.     
{The equivalent of assumption \eqref{eq:converseStSp} at the continuous level can be found in Lemma~\ref{lem:RCDconverseStSp}.}

When the local and the global continuous norms are related as above again, it is immediate to prove inequality \eqref{eq:finOvr} with 
\begin{equation*}
\Lambda_1 = \max \set{m | \exists \, j_1 \ne \dots \ne j_m \text{ such that } \textup{meas}(\Omega_{j_1} \cap \dots \cap \Omega_{j_m}) \ne 0},
\end{equation*}
that is $\Lambda_1$ is the maximal multiplicity of the subdomain intersection (this constant is like the one defined in \cite[Lemma 7.13]{DoJoNa:bookDDM} 
and is slightly more precise than $\Lambda_0$ that was used in \cite[eq.~(2.10)]{GrSpZo:impedance}). 
{The equivalent of assumption \eqref{eq:finOvr} at the continuous level can be found in Lemma~\ref{lem:RCDfinOvr}.}
Note that $\Lambda_0$ and $\Lambda_1$ are geometric constants, related to the decomposition into overlapping subdomains.  

{The remaining assumptions can be also expressed in the finite element language, which is introduced in the next section: see equation \eqref{eq:RCDstability} for the stability assumption \eqref{eq:stability}; equation \eqref{eq:RCDsim3-16} for assumption \eqref{eq:sim3-16} on the partition of unity;  equations \eqref{eq:RCD-cDF}, \eqref{eq:RCD-cDB} for assumptions \eqref{eq:cDF},\eqref{eq:cDB} on the commutators between the partition of unity and the local (inner product and problem) matrices.}

 

\section{The reaction-convection-diffusion equation} 
\label{sec:RCDeq}

As an illustration of the general theory, we apply Theorem~\ref{thm:main} to the case of the heterogeneous reaction-convection-diffusion equation; recall that the convergence theory for the homogeneous, respectively heterogeneous, Helmholtz equation was developed in \cite{GrSpZo:impedance}, respectively \cite{GoGrSp:heterHelm}.  
Let $\Omega \subset \mathbb{R}^d$ be an open bounded polyhedral domain. We study the heterogeneous reaction-convection-diffusion problem in conservative form, with Robin-type and Dirichlet boundary conditions: 
\begin{equation}
\label{eq:RCDbvp}
\begin{cases}
c_0 u + \dive(\mathbf{a}u) -\dive(\nu \nabla u) = f & \text{in } \Omega,\\ 
\nu \frac{\partial u}{\partial n} -\frac{1}{2} \mathbf{a}\cdot \mathbf{n} \, u + \alpha u = g & \text{on } \Gamma_R,\\
u = 0 & \text{on } \Gamma_D,
\end{cases}
\end{equation}
where $\partial \Omega = \Gamma = \Gamma_R \cup \Gamma_D$, 
$\mathbf{n}$ is the outward-pointing unit normal vector to $\Gamma$, 
{$c_0 \in \mathrm{L}^\infty(\Omega)$, $\mathbf{a} \in \mathrm{L}^\infty(\Omega)^d$, $\dive \mathbf{a} \in \mathrm{L}^\infty(\Omega)$, $\nu \in \mathrm{L}^\infty(\Omega)$, $f \in \mathrm{L}^2(\Omega)$, $g \in \mathrm{L}^2(\Gamma_R)$, $\alpha \in \mathrm{L}^\infty(\Omega)$ (in this case all quantities are real-valued). 
With the notation 
\[
\tilde{c} \coloneqq  c_0 + \frac{1}{2} \dive \mathbf{a}, 
\]
suppose that there exist $\tilde{c}_- > 0$, $\tilde{c}_+ > 0$ such that 
\begin{equation}
\label{eq:hypctilde}
\tilde{c}_- \le \tilde{c}(\mathbf{x}) \le \tilde{c}_+ \; \text{a.e. in } \Omega,
\end{equation}
(the positiveness of $\tilde{c}(\mathbf{x})$ is a classical assumption in reaction-convection-diffusion equation literature), and there exist $\nu_- > 0$, $\nu_+ > 0$ such that 
\[
\nu_- \le \nu(\mathbf{x}) \le \nu_+ \; \text{a.e. in } \Omega,
\]
and $\alpha(\mathbf{x}) \ge 0$ a.e.~in $\Omega$. }
Note that the appropriate Robin-type boundary condition (on $\Gamma_R$) here is not simply $\nu \frac{\partial u}{\partial n} + \alpha u = g$; we will comment below about a possible choice of $\alpha$, see \eqref{eq:exalpha}. 
Now, set $\mathrm{H}_{0,D}^1(\Omega) \coloneqq \set{v \in \mathrm{H}^1(\Omega) | v=0 \text{ on } \Gamma_D}$. 
In order to find the variational formulation, multiply the equation by a test function $v \in \mathrm{H}_{0,D}^1(\Omega)$ and integrate over $\Omega$:
\[
\int_\Omega \Bigl ( c_0 u v  + \frac{1}{2} \dive(\mathbf{a}u) v + \frac{1}{2} \dive(\mathbf{a}u)v -\dive(\nu \nabla u) \,v \Bigr) = \int_\Omega fv. 
\]
For the first divergence term use the identity $\dive(\mathbf{a}u) =  \dive(\mathbf{a})u+ \mathbf{a} \cdot \nabla u$, while for the second integrate by parts: 
\[
\int_\Omega \frac{1}{2} \dive(\mathbf{a}u) v = \int_\Omega - \frac{1}{2} u \, \mathbf{a} \cdot \nabla v + \int_{\partial \Omega} \frac{1}{2} \mathbf{a}\cdot \mathbf{n} \, u v,
\]
and, also by integration by parts, 
\[
\int_\Omega -\dive(\nu \nabla u) \,v = \int_\Omega \nu \nabla u \cdot \nabla v - \int_{\partial \Omega} \nu \frac{\partial u}{\partial n} v.
\]
Therefore, imposing the boundary conditions, the variational formulation is: find $u \in \mathrm{H}_{0,D}^1(\Omega)$ such that
\[
a(u,v) = F(v), \quad \text{for all } v \in \mathrm{H}_{0,D}^1(\Omega),
\]
where $a$ is a non-symmetric bilinear form defined as
\[
a(u,v) = \int_\Omega \Bigl ( \tilde{c} u v + \frac{1}{2}  \mathbf{a} \cdot \nabla u \, v - \frac{1}{2} u \, \mathbf{a} \cdot \nabla v  + \nu \nabla u \cdot \nabla v \Bigr ) + \int_{\Gamma_R} \alpha u v.
\]
and 
\[
F(v) \coloneqq  \int_\Omega fv +  \int_{\Gamma_R} g v.
\]
Define the weighted scalar product and norm
\[
(u,v)_{1,c} \coloneqq \int_\Omega \Bigl ( \tilde{c} u v+ \nu \nabla u \cdot \nabla v \Bigr ), \qquad \lVert u \rVert_{1,c} \coloneqq (u,u)_{1,c}^{1/2}.
\]
On each subdomain $\Omega_j$ we consider the local problem with bilinear form
\[
a_j(u,v) \coloneqq \int_{\Omega_j} \Bigl ( \tilde{c} u v + \frac{1}{2}  \mathbf{a} \cdot \nabla u \, v - \frac{1}{2} u \, \mathbf{a} \cdot \nabla v  + \nu \nabla u \cdot \nabla v \Bigr ) + \int_{\partial \Omega_j \setminus \Gamma_D} \alpha u v,
\]
where we impose absorbing transmission conditions on the subdomain interface $\partial\Omega_j \setminus \partial\Omega$: 
for instance, we can choose a zeroth-order Taylor approximation of transparent conditions given by 
\begin{equation}
\label{eq:exalpha}
\alpha = \sqrt{(\mathbf{a}\cdot\mathbf{n})^2+4c_0\nu}/2 
\end{equation}
(see e.g.~\cite{JaNaRo:2001:OO2} and the references therein).   
We define the local weighted scalar product and norm
\[
(u,v)_{1,c,\Omega_j} \coloneqq \int_{\Omega_j} \Bigl ( \tilde{c} u v+ \nu \nabla u \cdot \nabla v \Bigr ), 
\qquad \lVert u \rVert_{1,c, \Omega_j} \coloneqq (u,u)_{1,c, \Omega_j}^{1/2},
\]
which would correspond to Neumann-type boundary conditions on $\partial \Omega_j$. 
Set 
\begin{align*}
\tilde{c}_{+,j} &\coloneqq \lVert \tilde{c} \rVert_{\mathrm{L}^\infty(\Omega_j)}, & \tilde{c}_{-,j} &\coloneqq \lVert \tilde{c}^{-1} \rVert_{\mathrm{L}^\infty(\Omega_j)}^{-1}, & &\text{so } \tilde{c}_{-,j} \le \tilde{c}(\mathbf{x}) \le \tilde{c}_{+,j} \; \text{a.e. in } \Omega_j, \\
\nu_{+,j} &\coloneqq \lVert \nu \rVert_{\mathrm{L}^\infty(\Omega_j)}, & \nu_{-,j} &\coloneqq \lVert \nu^{-1} \rVert_{\mathrm{L}^\infty(\Omega_j)}^{-1}, 
& &\text{so } \nu_{-,j} \le \nu(\mathbf{x}) \le \nu_{+,j} \; \text{a.e. in } \Omega_j.
\end{align*}

\begin{remark}
\label{rem:globloc}
For $u,v \in \mathrm{H}^1(\Omega)$, if $u$ or $v$ are supported in $\overline{\Omega}_j$ and thus vanish on $\partial \Omega_j\setminus \partial \Omega$, then
\[
a(u,v) = a_j(u,v), \quad \text{and} \quad (u,v)_{1,c} = (u,v)_{1,c,\Omega_j}.
\] 
\end{remark}

For the finite element discretization, let $\mathcal{T}^h$ be a family of conforming simplicial meshes of $\Omega$ that are $h$-uniformly shape regular as the mesh diameter $h$ tends to zero. We consider finite elements of order $r$
\[
\mathcal{V}^h = \set{v_h \in C^0(\overline{\Omega}), v_h\vert_{\tau} \in \mathbb{P}_{r-1}(\tau) \: \forall \, \tau \in \mathcal{T}^h,  v_h\vert_{\Gamma_D} = 0} \subset \mathrm{H}_{0,D}^1(\Omega). 
\]
Consider nodal basis functions $\varphi_i, i=1,\dots,n$ (for example Lagrange basis functions), in duality with the degrees of freedom associated with nodes $\mathbf{x}_j, j=1,\dots,n$, that is $\varphi_i(\mathbf{x}_j) = \delta_{ij}$.  
Thus we can define the standard nodal Lagrange interpolation operator $\Pi^h v = \sum_{i=1}^n v(\mathbf{x}_i) \varphi_i$. 
Assume that $\mathcal{V}^h$ satisfies the standard interpolation error estimate (see e.g. \cite[§3.1]{Ciarlet:book:1978}): 
for $\tau \in \mathcal{T}^h$, provided $v \in \mathrm{H}^r(\tau)$ 
\begin{equation}
\label{eq:errEst}
\lVert (I - \Pi^h)v \rVert_{\mathrm{L}^2(\tau)} + h \lvert (I - \Pi^h)v \rvert_{\mathrm{H}^1(\tau)} \le C_\Pi h^r \lvert v \rvert_{\mathrm{H}^r(\tau)}.   
\end{equation}
Assume that the subdomains $\Omega_j$ are polyhedra with characteristic length scale $H_\textup{sub}$, which means 
\begin{definition}[Characteristic length scale]
\label{def:lengthscale}
A domain has characteristic length scale $L$ if its diameter $\sim L$, its surface area $\sim L^{d-1}$, and its volume $\sim L^d$, 
where $\sim$ means uniformly bounded from below and above.
\end{definition}
\noindent
For each $j=1,\dots,N$, denote by $\mathcal{V}_j^h$ the space of functions in $\mathcal{V}^h$ restricted to $\overline{\Omega}_j$.  
So, $A$, $F_\Omega$, $B_j$, $F_{\Omega_j}$ are defined as the matrices arising, respectively, from the finite element discretization of $a$, $(\cdot,\cdot)_{1,c}$ on $\mathcal{V}^h$, and $a_j$, $(\cdot,\cdot)_{1,c,\Omega_j}$ on $\mathcal{V}_j^h$: for $v_h, w_h \in \mathcal{V}^h$ with vectors of degrees of freedom $\mathbf{V}, \mathbf{W} \in \mathbb{R}^n$, and for $v_h^j, w_h^j \in \mathcal{V}_j^h$ with vectors of degrees of freedom $\mathbf{V}^j, \mathbf{W}^j \in \mathbb{R}^{n_j}$
\begin{align}
a(v_h,w_h) &= (A\mathbf{V},\mathbf{W}), & a_j(v_h^j,w_h^j) &= (B_j\mathbf{V}^j,\mathbf{W}^j), \label{eq:relsBilForms}\\
(v_h,w_h)_{1,c} &= (F_\Omega \mathbf{V},\mathbf{W}), & (v_h^j,w_h^j)_{1,c,\Omega_j} &= (F_{\Omega_j} \mathbf{V}^j,\mathbf{W}^j). \label{eq:relsNorms}
\end{align}

Consider partition of unity functions $\chi_j$, $j=1,\dots,N$, such that $\sum_{j=1}^N \chi_j = 1$ in $\overline{\Omega}$, and $\mathrm{supp}( \chi_j )\subset \Omega_j$, so in particular they are zero on $\partial \Omega_j \setminus \partial \Omega$. 
Assume that 
\begin{equation}
\label{eq:2-21}
\lVert \partial_\mathbf{x}^\beta \chi_j \rVert_{\infty,\tau} \le C_\textup{dPU} \frac{1}{\delta^{|\beta|}} \quad \text{ for all } \tau \in \mathcal{T}_h \; \text{and multi-index $\beta$ with $|\beta|\le r$},
\end{equation} 
where $\delta$ is the size of the overlap between subdomains, and $C_\textup{dPU}$ is required to be independent of the simplex $\tau$ and of the derivative multi-index $\beta$. 
The diagonal matrices $D_j$ are constructed by interpolation of the functions $\chi_j$, so the vector of degrees of freedom of $\Pi^h(\chi_j v_h)$ is $D_j R_j  \mathbf{V}$.

Next we need to introduce a technical ingredient, namely so-called multiplicative
trace inequalities. Such estimates can be found e.g.~in
\cite{Grisvard:book:1985}.
\begin{lemma}[{Multiplicative trace inequality, \cite[last eq. on page 41]{Grisvard:book:1985}}]
	\label{lem:MTI}
For any bounded Lipschitz open subset $\omega\subset \mathbb{R}^d$ there exists
    $C_\textup{tr}(\omega)>0$ such that, for all $u\in \mathrm{H}^{1}(\omega)$, we have
    $\Vert u\Vert_{\mathrm{L}^{2}(\partial\omega)}^{2}\leq C_\textup{tr}(\omega)(
    \Vert u\Vert_{\mathrm{L}^{2}(\omega)}\Vert \nabla u\Vert_{\mathrm{L}^{2}(\omega)}+
    \Vert u\Vert_{\mathrm{L}^{2}(\omega)}^{2}/\mathrm{diam}(\omega))$.
\end{lemma}
\noindent 
Although the constant $C_\textup{tr}(\omega)$ above does a priori depend on the shape of
$\omega$, it does not depend on its diameter (it is invariant under homothety). In the sequel
we shall assume that there exists a \textit{fixed} constant $C_\textup{tr}>0$ such that we have
$C_\textup{tr}(\Omega_{j})<C_\textup{tr}$. This holds for example if the subdomains are assumed to 
be uniformly star-shaped i.e. there exists a fixed constant $\mu>0$ such that, for each $j$ there
exists $\boldsymbol{x}_{\Omega_{j}}\in\Omega_{j}$ satisfying 
\begin{equation}
  \begin{aligned}
    &  \forall \boldsymbol{x}\in\partial\Omega_{j},\; \lbrack \boldsymbol{x},
    \boldsymbol{x}_{\Omega_{j}}\rbrack\subset \overline{\Omega}_{j}\quad \text{and} \\
    &\boldsymbol{n}_{j}(\boldsymbol{x})\cdot (\boldsymbol{x}- \boldsymbol{x}_{\Omega_{j}})
\geq \mu \vert \boldsymbol{x}- \boldsymbol{x}_{\Omega_{j}}\vert 
  \end{aligned}
\end{equation}
\begin{assumption}
The multiplicative trace estimates of Lemma~\ref{lem:MTI} hold uniformly for all subdomains.
\end{assumption}
This assumption allows to derive uniform upper bounds for the continuity modulus of the bilinear forms $a(\;,\;)$ and $a_{j}(\;,\;)$. 

\begin{lemma}[Continuity of the bilinear forms $a$ and $a_j$]
\label{lem:RCDcont}
Assume that $\Omega$ has characteristic length scale $L$ in the sense of Definition \ref{def:lengthscale}. 
Then for all $u, v \in \mathrm{H}^1(\Omega)$
\[
a(u,v) \le  
C_{\textup{cont}} 
\lVert u \rVert_{1,c} \lVert v \rVert_{1,c}, 
\]
where
\[
C_{\textup{cont}} =  \frac{\tilde{c}_+}{\tilde{c}_-} \frac{\nu_+}{\nu_-} +
\frac{1}{2}  \frac{\lVert \mathbf{a} \rVert_{\mathrm{L}^\infty(\Omega)}}{\sqrt{\nu_- \tilde{c}_-}} +
\frac{\lVert \alpha \rVert_{\mathrm{L}^\infty(\Omega)} C_\textup{tr}}{\sqrt{\tilde{c}_-}} 
\left( \frac{1}{L \sqrt{\tilde{c}_-}} + \frac{1}{2\sqrt{\nu_-}} \right).
\]
Similarly for all $u, v \in \mathrm{H}^1(\Omega_j)$
\begin{equation}
\label{eq:RCDloccont}
a_j(u,v) \le  
C_{\textup{cont},j} 
\lVert u \rVert_{1,c,\Omega_j} \lVert v \rVert_{1,c,\Omega_j},
\end{equation}
where
\begin{equation}
\label{eq:Ccontj}
C_{\textup{cont},j} = \frac{\tilde{c}_{+,j}}{\tilde{c}_{-,j}} \frac{\nu_{+,j}}{\nu_{-,j}} +
\frac{1}{2}  \frac{\lVert \mathbf{a} \rVert_{\mathrm{L}^\infty(\Omega_j)}}{\sqrt{\nu_{-,j} \tilde{c}_{-,j}}} +
\frac{\lVert \alpha \rVert_{\mathrm{L}^\infty(\Omega_j)} C_\textup{tr}}{\sqrt{\tilde{c}_{-,j}}} 
\left( \frac{1}{H_\textup{sub} \sqrt{\tilde{c}_{-,j}}} + \frac{1}{2\sqrt{\nu_{-,j}}} \right). 
\end{equation}
\end{lemma}
\begin{proof}
By Cauchy-Schwarz inequality
\[
\begin{split}
a(u,v) 
& \le \tilde{c}_+  \lVert u \rVert_{\mathrm{L}^2(\Omega)} \lVert v \rVert_{\mathrm{L}^2(\Omega)} 
+ \nu_+  \lVert \nabla u \rVert_{\mathrm{L}^2(\Omega)} \lVert \nabla v \rVert_{\mathrm{L}^2(\Omega)} \\ 
&+ \frac{1}{2}  \lVert \mathbf{a} \rVert_{\mathrm{L}^\infty(\Omega)}  \left( \lVert \nabla u \rVert_{\mathrm{L}^2(\Omega)} \lVert v \rVert_{\mathrm{L}^2(\Omega)} + \lVert u \rVert_{\mathrm{L}^2(\Omega)} \lVert \nabla v \rVert_{\mathrm{L}^2(\Omega)}  \right) \\
& + \lVert \alpha \rVert_{\mathrm{L}^\infty(\Omega)} \lVert u \rVert_{\mathrm{L}^2(\Gamma_R)} \lVert v \rVert_{\mathrm{L}^2(\Gamma_R)}. 
\end{split}
\]
First, using the Cauchy-Schwarz inequality with respect to the Euclidean inner product in $\mathbb{R}^2$ and $1 \le ({\tilde{c}_+}/{\tilde{c}_-})$, $1 \le ({\nu_+}/{\nu_-})$, we get 
\[
\begin{split}
&\tilde{c}_+  \lVert u \rVert_{\mathrm{L}^2(\Omega)} \lVert v \rVert_{\mathrm{L}^2(\Omega)} 
+ \nu_+  \lVert \nabla u \rVert_{\mathrm{L}^2(\Omega)} \lVert \nabla v \rVert_{\mathrm{L}^2(\Omega)} \\
&= 
\begin{pmatrix} \frac{\tilde{c}_+}{\tilde{c}_-} \sqrt{\tilde{c}_-} \lVert u \rVert_{\mathrm{L}^2(\Omega)} & \frac{\nu_+}{\nu_-} \sqrt{\nu_-} \lVert \nabla u \rVert_{\mathrm{L}^2(\Omega)} \end{pmatrix}
\begin{pmatrix} \sqrt{\tilde{c}_-} \lVert v \rVert_{\mathrm{L}^2(\Omega)} \\ \sqrt{\nu_-} \lVert \nabla v \rVert_{\mathrm{L}^2(\Omega)} \end{pmatrix} \\
& \le  \frac{\tilde{c}_+}{\tilde{c}_-} \frac{\nu_+}{\nu_-} \left( \tilde{c}_{-} \lVert u \rVert_{\mathrm{L}^2(\Omega)}^2 + \nu_- \lVert \nabla u \rVert_{\mathrm{L}^2(\Omega)}^2 \right)^{1/2} 
\left( \tilde{c}_{-} \lVert v \rVert_{\mathrm{L}^2(\Omega)}^2 + \nu_- \lVert \nabla v \rVert_{\mathrm{L}^2(\Omega)}^2 \right)^{1/2} \\
& \le \frac{\tilde{c}_+}{\tilde{c}_-} \frac{\nu_+}{\nu_-} \lVert u \rVert_{1,c} \lVert v \rVert_{1,c}.
\end{split}
\]
Second
\[
\begin{split}
& \lVert \nabla u \rVert_{\mathrm{L}^2(\Omega)} \lVert v \rVert_{\mathrm{L}^2(\Omega)} + \lVert u \rVert_{\mathrm{L}^2(\Omega)} \lVert \nabla v \rVert_{\mathrm{L}^2(\Omega)} \\
& = \frac{1}{\sqrt{\nu_- \tilde{c}_-}}
\begin{pmatrix} \sqrt{\nu_-} \lVert \nabla u \rVert_{\mathrm{L}^2(\Omega)} & \sqrt{\tilde{c}_-} \lVert u \rVert_{\mathrm{L}^2(\Omega)}\end{pmatrix}
\begin{pmatrix} \sqrt{\tilde{c}_-} \lVert v \rVert_{\mathrm{L}^2(\Omega)} \\ \sqrt{\nu_-} \lVert \nabla v \rVert_{\mathrm{L}^2(\Omega)} \end{pmatrix} \\
&\le \frac{1}{\sqrt{\nu_- \tilde{c}_-}} 
\left( \tilde{c}_{-} \lVert u \rVert_{\mathrm{L}^2(\Omega)}^2 + \nu_- \lVert \nabla u \rVert_{\mathrm{L}^2(\Omega)}^2 \right)^{1/2} 
\left( \tilde{c}_{-} \lVert v \rVert_{\mathrm{L}^2(\Omega)}^2 + \nu_- \lVert \nabla v \rVert_{\mathrm{L}^2(\Omega)}^2 \right)^{1/2} \\
&\le \frac{1}{\sqrt{\nu_- \tilde{c}_-}} \lVert u \rVert_{1,c} \lVert v \rVert_{1,c}.
\end{split}
\]
Third, for the boundary term, using the multiplicative trace inequality recalled in Lemma~\ref{lem:MTI} and using also the inequality $ab \le (a^2+b^2)/2$ valid for all $a,b>0$, we have
\[
\begin{split}
&\lVert u \rVert_{\mathrm{L}^2(\Gamma_R)} \\
&\le \sqrt{C_\textup{tr}} \frac{1}{\sqrt[4]{\tilde{c}_-}} 
\left( \frac{1}{L \sqrt{\tilde{c}_-}} {\tilde{c}_-} \lVert u \rVert_{\mathrm{L}^2(\Omega)}^2 + \frac{1}{\sqrt{\nu_-}}\sqrt{\nu_-}\lVert \nabla u \rVert_{\mathrm{L}^2(\Omega)} \sqrt{\tilde{c}_-} \lVert u \rVert_{\mathrm{L}^2(\Omega)} \right)^{1/2} \\
& \le \sqrt{C_\textup{tr}} \frac{1}{\sqrt[4]{\tilde{c}_-}} 
\left( \frac{1}{L \sqrt{\tilde{c}_-}} \lVert u \rVert_{1,c}^2 + \frac{1}{2\sqrt{\nu_-}}  \lVert u \rVert_{1,c}^2 \right)^{1/2} \\
&= \sqrt{C_\textup{tr}} \frac{1}{\sqrt[4]{\tilde{c}_-}} 
\left( \frac{1}{L \sqrt{\tilde{c}_-}} + \frac{1}{2\sqrt{\nu_-}} \right)^{1/2}  \lVert u \rVert_{1,c}
\end{split}
\]
and
\[
\lVert \alpha \rVert_{\mathrm{L}^\infty(\Omega)} \lVert u \rVert_{\mathrm{L}^2(\Gamma_R)} \lVert v \rVert_{\mathrm{L}^2(\Gamma_R)}  
\le \lVert \alpha \rVert_{\mathrm{L}^\infty(\Omega)} {C_\textup{tr}} \frac{1}{\sqrt{\tilde{c}_-}} 
\biggl( \frac{1}{L \sqrt{\tilde{c}_-}} + \frac{1}{2\sqrt{\nu_-}} \biggr)  \lVert u \rVert_{1,c} \lVert v \rVert_{1,c}.
\]
In conclusion
\[
a(u,v) \le 
\biggl ( \frac{\tilde{c}_+}{\tilde{c}_-} \frac{\nu_+}{\nu_-} +
\frac{1}{2}  \frac{\lVert \mathbf{a} \rVert_{\mathrm{L}^\infty(\Omega)}}{\sqrt{\nu_- \tilde{c}_-}} +
\frac{\lVert \alpha \rVert_{\mathrm{L}^\infty(\Omega)} C_\textup{tr}}{\sqrt{\tilde{c}_-}} 
\biggl( \frac{1}{L \sqrt{\tilde{c}_-}} + \frac{1}{2\sqrt{\nu_-}} \biggr)  \biggr )
\lVert u \rVert_{1,c} \lVert v \rVert_{1,c}.
\]

Finally, note that the local bilinear form $a_j$ has the same form as the bilinear form $a$, so the analogous inequality holds (with $L=H_\textup{sub}$). 
\end{proof}

\begin{lemma}[Coercivity of the bilinear forms $a$ and $a_j$]
\label{lem:RCDcoerc}
We have
\begin{gather}
a(v,v) \ge  \lVert v \rVert_{1,c}^2 \quad \text{for all } v \in \mathrm{H}^1(\Omega), \label{eq:RCDcoercGl}\\
a_j(v,v) \ge  \lVert v \rVert_{1,c,\Omega_j}^2 \quad \text{for all } v \in \mathrm{H}^1(\Omega_j) \label{eq:RCDcoercLoc}.
\end{gather}
\end{lemma}
\begin{proof}
Note that
\[
a(v,v) = \int_\Omega \Bigl ( \tilde{c} v^2 + \nu |\nabla v|^2 \Bigr ) + \int_{\Gamma_R} \alpha v^2,
\]
and 
\[
a_j(v,v) =  \int_{\Omega_j} \Bigl ( \tilde{c} v^2 + \nu |\nabla v|^2  \Bigr ) + \int_{\partial \Omega_j \setminus \Gamma_D} \alpha v^2,
\]
because the anti-symmetric terms cancel out. 
Thus properties \eqref{eq:RCDcoercGl}-\eqref{eq:RCDcoercLoc} follow.
\end{proof}
\noindent
Note that the good constant in the coercivity estimates is a result of careful choices made in the derivation of the bilinear forms (see the beginning of section \ref{sec:RCDeq}), such as the handling of the $\dive(\mathbf{a}u)v$ term (split into two parts with different treatments) and the definition of suitable Robin-type boundary conditions.

\subsection{Estimates for the assumptions of Theorem \ref{thm:main}}
Now we prove, for the heterogeneous reaction-convection-diffusion problem \eqref{eq:RCDbvp}, the equalities and inequalities that have been identified in Theorem \ref{thm:main} as the assumptions for the convergence analysis. 
In the proofs we do not make any assumption on the regime of the physical coefficients of the equation nor of the numerical parameters. 
{Note that to prove the stability assumption \eqref{eq:stability} for this problem we have considered the case of a coercive bilinear form (which is non-symmetric, though), see Lemma~\ref{lem:RCDstab}. However, in general the problem does not need to be positive definite for   Theorem \ref{thm:main} to be valid.}  

In what follows, we prove equalities and estimates in the continuous setting, which can be translated into results in the discrete setting recalling relations \eqref{eq:relsBilForms} between the continuous and discrete bilinear forms, relations \eqref{eq:relsNorms} between the continuous and discrete inner products (hence between the norms), and the fact that the vector of degrees of freedom of $\Pi^h(\chi_j v_h)$ is $D_j R_j  \mathbf{V}$. 

First of all, note that the partition of unity, the global and local bilinear forms and norms fit the typical framework identified in §\ref{subsec:commentsHp}: {the entries corresponding to $\partial \Omega_j \setminus \partial \Omega$ of the partition of unity matrix $D_j$ are zero; the local bilinear form is like the global bilinear form but with the integrals on $\Omega_j$ instead of $\Omega$ and with an additional boundary integral on $\partial \Omega_j \setminus \partial \Omega$; the local norm can be obtained from the global norm just by replacing $\Omega$ with $\Omega_j$ in the integration domain.  
Therefore it is not surprising that assumptions  \eqref{eq:DAB}, \eqref{eq:FkD},  \eqref{eq:converseStSp}, \eqref{eq:finOvr} are verified. As a more precise illustration of the general remarks in §\ref{subsec:commentsHp}, we first provide the detailed proof of assumptions \eqref{eq:DAB} and \eqref{eq:FkD}, which is essentially based on Remark~\ref{rem:globloc}:}

\begin{lemma}
\label{lem:hypglobloc}
For all global vectors of degrees of freedom $\mathbf{U} \in \mathbb{R}^n$ and local vectors of degrees of freedom $\mathbf{V}^j \in \mathbb{R}^{n_j}$ in $\Omega_j$, $j=1,\dots,N$, we have
\begin{align*}
(D_j R_j A \mathbf{U}, \mathbf{V}^j) &= (D_j B_j R_j \mathbf{U}, \mathbf{V}^j),\\
(D_j R_j F_\Omega \mathbf{U}, \mathbf{V}^j) &= (D_j F_{\Omega_j} R_j \mathbf{U}, \mathbf{V}^j).
\end{align*}
\end{lemma}
\begin{proof}
Since the partition of unity matrices $D_j$ are diagonal, hence symmetric, and the restriction matrices $R_j$ satisfy $(\mathbf{V},R_j^T \mathbf{W}^j) = (R_j \mathbf{V}, \mathbf{W}^j)$ and $R_j R_j^T \mathbf{V}^j = \mathbf{V}^j$, we can write
\[
(D_j R_j A \mathbf{U}, \mathbf{V}^j) = (A \mathbf{U}, R_j^T D_j \mathbf{V}^j) = (A \mathbf{U}, R_j^T D_j R_j R_j^T \mathbf{V}^j). 
\] 
Now, call $\widetilde{\mathbf{V}}^j \coloneqq R_j^T \mathbf{V}^j$ and $\widetilde{v}_j \in \mathcal{V}^h$ the function with degrees of freedom given by $\widetilde{\mathbf{V}}^j$, so $D_j R_j R_j^T \mathbf{V}^j$ is the local vector of degrees of freedom of $\Pi^h(\chi_j \widetilde{v}_j)$, and $R_j^T D_j R_j R_j^T \mathbf{V}^j$ is the global vector of degrees of freedom of $\Pi^h(\chi_j \widetilde{v}_j)$. Call $u \in \mathcal{V}^h$ the function with degrees of freedom given by $\mathbf{U}$. Therefore
\[
(A \mathbf{U}, R_j^T D_j R_j R_j^T \mathbf{V}^j) = a(u,\Pi^h(\chi_j \widetilde{v}_j)). 
\]
Moreover, observe that $\chi_j \widetilde{v}_j$ is supported in $\Omega_j$ and vanishes on $\partial \Omega_j \setminus \partial \Omega$, thus the same is true for its interpolant $\Pi^h(\chi_j \widetilde{v}_j)$, and by applying Remark~\ref{rem:globloc} we obtain 
\[
a(u,\Pi^h(\chi_j \widetilde{v}_j)) = a_j(u,\Pi^h(\chi_j \widetilde{v}_j)).
\]
Finally
\[
a_j(u,\Pi^h(\chi_j \widetilde{v}_j)) =  (B_j R_j \mathbf{U}, D_j R_j R_j^T \mathbf{V}^j) = (B_j R_j \mathbf{U}, D_j \mathbf{V}^j) = (D_j B_j R_j \mathbf{U}, \mathbf{V}^j).
\]
The proof of $(D_j R_j F_\Omega \mathbf{U}, \mathbf{V}^j) = (D_j F_{\Omega_j} R_j \mathbf{U}, \mathbf{V}^j)$ proceeds in the same way.  
\end{proof}

{
Now we prove that assumptions \eqref{eq:converseStSp} and \eqref{eq:finOvr} are indeed verified with the geometric constants $\Lambda_0$ and $\Lambda_1$ defined in §\ref{subsec:commentsHp}:
\begin{lemma}[Continuous version of assumption \eqref{eq:converseStSp}]
\label{lem:RCDconverseStSp}
For all $w^j \in \mathrm{H}^1(\Omega_j)$, $j=1,\dots,N$, denoting by $\widetilde{w}^j$ their extensions by zero to $\Omega$,  we have
\begin{equation*}
\biggl \lVert  \sum_{j=1}^N \widetilde{w}^j \biggr \rVert_{1,c}^2 \le 
\Lambda_0 \sum_{j=1}^N \lVert w^j \rVert_{1,c,\Omega_j}^2,
\end{equation*}
where $\Lambda_0$ is the maximum number of neighboring  subdomains:
\begin{equation*}
\Lambda_0 = \max_{j=1,\dots,N} \, \# \Lambda(j), \quad \text{where } \Lambda(j) \coloneqq \set{i | \Omega_j \cap \Omega_{i} \ne \emptyset}.  
\end{equation*}
\end{lemma}
\begin{proof}
By applying several times the Cauchy-Schwarz inequality (first for the scalar product $(\;,\;)_{1,c}$ and then twice for the dot product of the Euclidean space), and by Remark~\ref{rem:globloc}, we get 
\[
\begin{split}
&\biggl \lVert  \sum_{j=1}^N \widetilde{w}^j \biggr \rVert_{1,c}^2 
= \Biggl( \sum_{j=1}^N \widetilde{w}^j, \sum_{j'=1}^N \widetilde{w}^{j'} \Biggr)_{1,c}
= \sum_{j=1}^N \sum_{j'\in\Lambda(j)} (\widetilde{w}^{j}, \widetilde{w}^{j'})_{1,c} \\
&\le \sum_{j=1}^N \Biggl( \lVert w^j \rVert_{1,c,\Omega_j} \sum_{j'\in\Lambda(j)}  \lVert w^{j'} \rVert_{1,c,\Omega_{j'}} \Biggr)  \\
&\le \Biggl( \sum_{j=1}^N  \lVert w^j \rVert_{1,c,\Omega_j}^2 \Biggr)^{1/2} \Biggl( \sum_{j=1}^N \Bigl( \sum_{j'\in\Lambda(j)}  \lVert w^{j'} \rVert_{1,c,\Omega_{j'}} \Bigr)^2 \Biggl)^{1/2} \\
&\le \Biggl( \sum_{j=1}^N  \lVert w^j \rVert_{1,c,\Omega_j}^2 \Biggr)^{1/2} \Biggl( \sum_{j=1}^N \Bigl( \sum_{j'\in\Lambda(j)} 1^2 \sum_{j'\in\Lambda(j)}  \lVert w^{j'} \rVert_{1,c,\Omega_{j'}}^2 \Bigr) \Biggl)^{1/2}. 
\end{split}
\]
Now we have 
\[
\begin{split}
&\sum_{j=1}^N \Bigl( \sum_{j'\in\Lambda(j)} 1^2 \sum_{j'\in\Lambda(j)}  \lVert w^{j'} \rVert_{1,c,\Omega_{j'}}^2 \Bigr) 
= \sum_{j=1}^N \Bigl( \#\Lambda(j) \sum_{j'\in\Lambda(j)}  \lVert w^{j'} \rVert_{1,c,\Omega_{j'}}^2 \Bigr) \\
&\le \Lambda_0 \sum_{j=1}^N \sum_{j'\in\Lambda(j)}  \lVert w^{j'} \rVert_{1,c,\Omega_{j'}}^2 
= \Lambda_0 \sum_{j'=1}^N  \#\Lambda(j') \lVert w^{j'} \rVert_{1,c,\Omega_{j'}}^2 
\le \Lambda_0^2 \sum_{j'=1}^N \lVert w^{j'} \rVert_{1,c,\Omega_{j'}}^2.  
\end{split}
\]
Therefore in summary 
\[
\biggl \lVert  \sum_{j=1}^N \widetilde{w}^j \biggr \rVert_{1,c}^2 
\le \Biggl( \sum_{j=1}^N  \lVert w^j \rVert_{1,c,\Omega_j}^2 \Biggr)^{1/2}  \Biggl( \Lambda_0^2 \sum_{j'=1}^N \lVert w^{j'} \rVert_{1,c,\Omega_{j'}}^2 \Biggr)^{1/2} 
 = \Lambda_0 \sum_{j=1}^N \lVert w^j \rVert_{1,c,\Omega_j}^2.
\]
\end{proof}
}

{
\begin{lemma}[Continuous version of assumption \eqref{eq:finOvr}]
\label{lem:RCDfinOvr}
For all $v \in \mathrm{H}^1(\Omega)$
\begin{equation*}
 \sum_{j=1}^N \lVert v|_{\Omega_j} \rVert_{1,c,\Omega_j}^2 
 \le \Lambda_1 \Vert v \rVert_{1,c}^2,
\end{equation*}
where $\Lambda_1$ is the maximal multiplicity of the subdomain intersection:
\begin{equation*}
\Lambda_1 = \max \set{m | \exists \, j_1 \ne \dots \ne j_m \text{ such that } \textup{meas}(\Omega_{j_1} \cap \dots \cap \Omega_{j_m}) \ne 0}.
\end{equation*}
\end{lemma}
\begin{proof}
The result is immediate by the definition of the norms and of $\Lambda_1$:
\[
 \sum_{j=1}^N \lVert v|_{\Omega_j} \rVert_{1,c,\Omega_j}^2  = 
\sum_{j=1}^N  \int_{\Omega_j} \Bigl ( \tilde{c} (v|_{\Omega_j})^2+ \nu \bigl \lvert \nabla (v|_{\Omega_j}) \bigr \rvert^2 \Bigr )
\le \Lambda_1 \int_{\Omega} \Bigl ( \tilde{c} v^2+ \nu \lvert \nabla v \rvert^2 \Bigr ).
\]
\end{proof}
}

For the remaining assumptions, for the translation from the continuous to the discrete setting we also need to consider the error in interpolation of $\chi_j v_h$, studied in the following lemma.

\begin{lemma}
\label{lem:errInterpChi}
For any $j=1,\dots,N$, let $v_h \in \mathcal{V}^h_j$. 
Then
\begin{equation}
\label{eq:RCDsim3-10bis}
\lVert (\mathrm{I} - \Pi^h)(\chi_j v_h) \rVert_{1,c, \Omega_j} \le C_{\textup{err},j}  \lVert v_h \rVert_{1,c, \Omega_j}, 
\end{equation} 
where
\begin{equation}
\label{eq:Cerrj}
C_{\textup{err},j} = C_\Pi \, c(r,d) \, C_\textup{dPU} \, \sqrt{C_\textup{inv}} \left (\sqrt{\frac{\nu_{+,j}}{\nu_{-,j}}} + \sqrt{\frac{\tilde{c}_{+,j}}{\nu_{-,j}}} h \right) \frac{h}{\delta},
\end{equation} 
and $C_\Pi$ appears in \eqref{eq:errEst}, $C_\textup{dPU}$ in \eqref{eq:2-21}, $C_\textup{inv}$ is a standard inverse inequality constant (see the proof for more details),
and $c(r,d) = \max_{\vert \gamma\vert = r}\sum_{\beta \,|\, 0 < \beta\le\gamma}\binom{\gamma}{\beta}$.
\end{lemma}
\begin{proof}
For each simplex $\tau \in \mathcal{T}^h$, $\tau \subset \Omega_j$, from \eqref{eq:errEst} we have 
\begin{equation}
\label{eq:errEstchi}
\lVert (I - \Pi^h)(\chi_j v_h) \rVert_{\mathrm{L}^2(\tau)} + h \lvert (I - \Pi^h)(\chi_j v_h) \rvert_{\mathrm{H}^1(\tau)} \le C_\Pi h^r \lvert \chi_j v_h \rvert_{\mathrm{H}^r(\tau)}.   
\end{equation}
In order to estimate $\lvert \chi_j v_h \rvert_{\mathrm{H}^r(\tau)}$, let $\gamma \in \mathbb{N}^d$ be a multi-index of order $r$, i.e. $\lvert \gamma \rvert = r$. 
By the multivariate Leibniz rule and observing that $\partial_\mathbf{x}^\gamma v_h = 0$ since $v_h|_\tau$ is a polynomial of degree $r-1$, we have
\[
\partial_\mathbf{x}^\gamma (\chi_jv_h) = \sum_{\beta \,|\, 0\le\beta\le\gamma}\binom{\gamma}{\beta}(\partial_\mathbf{x}^\beta\chi_j)(\partial_\mathbf{x}^{\gamma-\beta}v_h) 
= \sum_{\beta \,|\, 0 < \beta\le\gamma}\binom{\gamma}{\beta}(\partial_\mathbf{x}^\beta\chi_j)(\partial_\mathbf{x}^{\gamma-\beta}v_h), 
\]
(note that in the last equality the multi-index $0 = (0, \dots, 0) \in \mathbb{N}^d$ is excluded).
Then, setting $c(r,d) = \max_{\vert \gamma\vert = r}\sum_{\beta \,|\, 0 < \beta\le\gamma}\binom{\gamma}{\beta}$, and using \eqref{eq:2-21}, we get
\begin{equation}
\label{eq:Dalpha}
\Vert \partial_\mathbf{x}^\gamma (\chi_jv_h) \rVert_{\mathrm{L}^2(\tau)} \le c(r,d) C_\textup{dPU} \max_{\beta \,|\, 0<\beta\le\gamma}\delta^{-\lvert \beta \rvert} \lvert v_h \rvert_{\mathrm{H}^{r-\lvert \beta \rvert}(\tau)}.
\end{equation}

Now we want to estimate $\lvert v_h \rvert_{\mathrm{H}^{r-\lvert \beta \rvert}(\tau)}$ using an inverse inequality, but in terms of the weighted norm $\lVert \; \rVert_{1,c,\tau}$ instead of the standard $\lVert \; \rVert_{H^1(\tau)}$ norm, and without making regime assumptions on the coefficients of the equation.   
First of all, note that, 
performing the change of variables $\mathbf{y}= \sqrt{\frac{\tilde{c}_{-,j}}{\nu_{-,j}}} \mathbf{x}$ and setting 
\[
\tau_c \coloneqq \biggl \{ \sqrt{\frac{\tilde{c}_{-,j}}{\nu_{-,j}}}\mathbf{x} \,\bigg |\, \mathbf{x} \in \tau \biggr \}, \quad 
\phi_c(v_h)(\mathbf{y}) \coloneqq v_h(\mathbf{x}) = v_h\left(\mathbf{y}\sqrt{\frac{\nu_{-,j}}{\tilde{c}_{-,j}}}\right),
\]
we can rewrite 
\begin{equation}
\label{eq:normschange}
\begin{split}
\lVert v_h \rVert_{1,c,\tau}^2
&\ge \int_\tau \left( \tilde{c}_{-,j} v_h ^2 + \nu_{-,j} \lvert \nabla_{\mathbf{x}} v_h|^2 \right) d\mathbf{x} \\
&=  \int_{\tau_c} \left(\tilde{c}_{-,j} (\phi_c(v_h))^2 + \nu_{-,j} \frac{\tilde{c}_{-,j}}{\nu_{-,j}} \lvert \nabla_{\mathbf{y}} \phi_c(v_h)|^2 \right) \biggl(\sqrt{\frac{\tilde{c}_{-,j}}{\nu_{-,j}}}\biggr)^{-d} d\mathbf{y} \\
&= \nu_{-,j} \left(\frac{\tilde{c}_{-,j}}{\nu_{-,j}}\right)^{1-d/2}  \lVert \phi_c(v_h) \rVert_{\mathrm{H}^1(\tau_c)}^2.
\end{split}
\end{equation}
Performing the same change of variables, we examine $\lvert v_h \rvert_{\mathrm{H}^{r-\lvert \beta \rvert}(\tau)}$: 
\[
\begin{split}
\lvert v_h \rvert_{\mathrm{H}^{r-\lvert\beta\rvert}(\tau)}^2 
&= \sum_{\xi \,|\, \lvert \xi \rvert =  r-\lvert\beta\rvert} \int_\tau \lvert \partial_\mathbf{x}^\xi v_h \rvert^2 d\mathbf{x} \\
&= \sum_{\xi \,|\, \lvert \xi \rvert =  r-\lvert\beta\rvert} \int_{\tau_c} \left(\frac{\tilde{c}_{-,j}}{\nu_{-,j}}\right)^{r-\lvert\beta\rvert}\lvert \partial_\mathbf{y}^\xi \phi_c(v_h) \rvert^2 \biggl(\sqrt{\frac{\tilde{c}_{-,j}}{\nu_{-,j}}}\biggr)^{-d}{d\mathbf{y}}\\
& = \left(\frac{\tilde{c}_{-,j}}{\nu_{-,j}}\right)^{r-\lvert\beta\rvert-d/2} \lvert \phi_c (v_h) \rvert_{\mathrm{H}^{r-\lvert\beta\rvert}(\tau_c)}^2,  
\end{split}
\]
so, using a standard inverse inequality (see e.g. \cite[Theorem 3.2.6]{Ciarlet:book:1978}), applied with $\sqrt{\frac{\tilde{c}_{-,j}}{\nu_{-,j}}}h$ (diameter of $\tau_c$), we get 
 \[
\begin{split}
\lvert v_h \rvert_{\mathrm{H}^{r-\lvert\beta\rvert}(\tau)}^2 
&\le C_\textup{inv}  \left(\frac{\tilde{c}_{-,j}}{\nu_{-,j}}\right)^{r-\lvert\beta\rvert-d/2} \biggl(\sqrt{\frac{\tilde{c}_{-,j}}{\nu_{-,j}}}h\biggr)^{-2(r-\lvert \beta \rvert -1)} \lVert \phi_c (v_h) \rVert_{\mathrm{H}^{1}(\tau_c)}^2 \\
& = C_\textup{inv} \left(\frac{\tilde{c}_{-,j}}{\nu_{-,j}}\right)^{1-d/2} h^{-2(r-\lvert \beta \rvert -1)} \lVert \phi_c (v_h) \rVert_{\mathrm{H}^{1}(\tau_c)}^2 \\
& \le C_\textup{inv} h^{-2(r-\lvert \beta \rvert -1)} \frac{1}{\nu_{-,j}} \lVert v_h \rVert_{1,c,\tau}^2,
\end{split}
\]
where the last inequality comes from \eqref{eq:normschange} (reversed). 
Therefore \eqref{eq:Dalpha} becomes:
\begin{equation}
\label{eq:Dalphafinal}
\begin{split}
\Vert \partial_\mathbf{x}^\gamma (\chi_jv_h) \rVert_{\mathrm{L}^2(\tau)} &\le 
c(r,d) C_\textup{dPU} \sqrt{C_\textup{inv}} \max_{m=1,\dots,r}\delta^{-m} h^{-(r- m  -1)}  \frac{1}{\sqrt{\nu_{-,j}}} \lVert v_h \rVert_{1,c,\tau} \\
& = c(r,d) C_\textup{dPU} \sqrt{C_\textup{inv}} \delta^{-1} h^{-r+2}  \frac{1}{\sqrt{\nu_{-,j}}} \lVert v_h \rVert_{1,c,\tau}, 
\end{split}
\end{equation}
where we have used the fact that $(h/\delta)\le 1$, so that the maximum is attained for $m=1$. 

Finally, combining \eqref{eq:errEstchi} and \eqref{eq:Dalphafinal}, and summing over all simplices $\tau \subset \Omega_j$, we obtain
\begin{align}
\lVert (\mathrm{I} - \Pi^h)(\chi_j v_h) \rVert_{\mathrm{L}^2(\Omega_j)} &\le C_\Pi c(r,d) C_\textup{dPU} \sqrt{C_\textup{inv}} \frac{h^2}{\delta} \frac{1}{\sqrt{\nu_{-,j}}} \lVert v_h \rVert_{1,c,\Omega_j},\label{eq:funcnorm}\\
\lvert (\mathrm{I} - \Pi^h)(\chi_j v_h) \rvert_{\mathrm{H}^1(\Omega_j)} &\le C_\Pi c(r,d) C_\textup{dPU} \sqrt{C_\textup{inv}} \frac{h}{\delta} \frac{1}{\sqrt{\nu_{-,j}}} \lVert v_h \rVert_{1,c,\Omega_j}. \label{eq:gradientnorm}
\end{align}
Now, applying $\sqrt{a^2+b^2} \le a+b$ with $a$ the left-hand side of \eqref{eq:funcnorm} multiplied by $\sqrt{\tilde{c}_{+,j}}$ and $b$ the left-hand side of \eqref{eq:gradientnorm} multiplied by $\sqrt{\nu_{+,j}}$ in order to recover the weighted norm, we obtain 
\[
\lVert (\mathrm{I} - \Pi^h)(\chi_j v_h) \rVert_{1,c, \Omega_j} \le C_\Pi c(r,d) C_\textup{dPU} \sqrt{C_\textup{inv}} \bigl (\sqrt{\tilde{c}_{+,j}} h + \sqrt{\nu_{+,j}} \bigr) \frac{h}{\delta} \frac{1}{\sqrt{\nu_{-,j}}} \lVert v_h \rVert_{1,c,\Omega_j}.
\]
\end{proof}

We prove now the stability bound \eqref{eq:stability}. 
\begin{lemma}(Stability bound for the local problems)
\label{lem:RCDstab}
For all $u_h^j \in \mathcal{V}_j^h$, we have
\begin{equation}
\label{eq:RCDstability}
\lVert u_h^j \rVert_{1,c,\Omega_j} \le \sup_{v_h^j \in \mathcal{V}_j^h \setminus \{0\}} \left( \frac{| a_j(u_h^j,v_h^j) |}{\lVert v_h^j \rVert_{1,c,\Omega_j}} \right). 
\end{equation}
Therefore, recalling the relation in \eqref{eq:relsBilForms} between the local continuous and discrete bilinear forms, assumption \eqref{eq:stability} is satisfied with 
\[
C_{\textup{stab},j} =1.
\]
\end{lemma}
\begin{proof}
This is a consequence of Lemmas \ref{lem:RCDcont}--\ref{lem:RCDcoerc} and Lax-Milgram theorem (see e.g. \cite[Theorem 5.14]{Spence:ibps:2015}): note that the constant in the stability bound is the reciprocal of the constant in the coercivity bound \eqref{eq:RCDcoercLoc}, which is $1$.  
\end{proof}
\noindent
The good constant obtained in the stability estimate 
is a result of careful choices made in the derivation of the bilinear form, as already pointed out for the coercivity estimate \eqref{eq:RCDcoercLoc}. 

Next, we prove estimates for assumption  \eqref{eq:sim3-16}. 
\begin{lemma}[$C_{D,j}$ in \eqref{eq:sim3-16}]
For all $v \in \mathrm{H}^1(\Omega_j)$ 
\begin{equation}
\label{eq:RCDsim3-1}
\lVert \chi_j v \rVert_{1,c, \Omega_j} \le  \sqrt{2} \left ( 1 + C_\textup{dPU} \sqrt{\frac{\nu_{+,j}}{\tilde{c}_{-,j}}} \frac{1}{\delta} \right )  \lVert v \rVert_{1,c, \Omega_j}, 
\end{equation}
where $C_\textup{dPU}$ appears in \eqref{eq:2-21}.
Moreover, for all $v_h \in \mathcal{V}^h_j$, 
{
\begin{equation}
\label{eq:RCDsim3-16}
\lVert \Pi^h(\chi_j v_h) \rVert_{1,c, \Omega_j}   
\le C_{D,j}
\lVert v_h \rVert_{1,c, \Omega_j}, 
\end{equation}
which is the finite element expression of assumption~\eqref{eq:sim3-16}, where 
\begin{equation}
\label{eq:RCDsim3-16const}
C_{D,j} = \sqrt{2} \left ( 1 + C_\textup{dPU} \sqrt{\frac{\nu_{+,j}}{\tilde{c}_{-,j}}} \frac{1}{\delta} \right )  + C_{\textup{err},j},
\end{equation}
}
with $ C_{\textup{err},j}$ given by \eqref{eq:Cerrj}.
\end{lemma} 
\begin{proof}
We have
\[
\lVert \chi_j v \rVert_{1,c, \Omega_j}^2 \le 
\int_{\Omega_j} \tilde{c} \lvert \chi_j v \rvert^2 + 2 \int_{\Omega_j} \nu |(\nabla \chi_j)v|^2 + 2 \int_{\Omega_j} \nu |\chi_j \nabla v|^2
\]
and using $|\chi_j| \le 1$ and \eqref{eq:2-21} we get
\[
\begin{split}
\lVert \chi_j v \rVert_{1,c, \Omega_j}^2 & \le
\int_{\Omega_j} \tilde{c} \lvert v \rvert^2 + 2 \int_{\Omega_j} \nu \, C_\textup{dPU}^2 \frac{1}{\delta^2} |v|^2 + 2 \int_{\Omega_j} \nu |\nabla v|^2\\
& \le 2 \Bigl ( 1 + C_\textup{dPU}^2 \frac{\nu_{+,j}}{\tilde{c}_{-,j}} \frac{1}{\delta^2} \Bigr ) \lVert v \rVert_{1,c, \Omega_j}^2.
\end{split}
\]

Now, for the second estimate, using the triangle inequality, the newly found inequality \eqref{eq:RCDsim3-1} and \eqref{eq:RCDsim3-10bis}, we get
\[
\begin{split}
\lVert \Pi^h(\chi_j v_h) \rVert_{1,c, \Omega_j} &\le \lVert \chi_j v_h \rVert_{1,c, \Omega_j}  + \lVert (\mathrm{I} - \Pi^h)(\chi_j v_h) \rVert_{1,c, \Omega_j}\\
& \le \left [ \sqrt{2} \left ( 1 + C_\textup{dPU} \sqrt{\frac{\nu_{+,j}}{\tilde{c}_{-,j}}} \frac{1}{\delta} \right )  + C_{\textup{err},j}  \right] \lVert v_h \rVert_{1,c, \Omega_j}.
\end{split}
\]
\end{proof} 

Next, we prove estimates for assumption \eqref{eq:cDF}, which involves a commutator between the partition of unity and the local inner product matrix. 

\begin{lemma}[{$C_{DF,j}$ in \eqref{eq:cDF}}]
\label{lem:RCDsimLemma3-7i}
For all $v,w \in \mathrm{H}^1(\Omega_j)$ 
\begin{equation}
\label{eq:RCDsimLemma3-7i}
\lvert (v,\chi_j w)_{1,c,\Omega_j} - (\chi_j v,w)_{1,c,\Omega_j}  \rvert \le  C_\textup{dPU}  \frac{\nu_{+,j}}{\sqrt{\tilde{c}_{-,j}\nu_{-,j}}} \frac{1}{\delta}   \lVert v \rVert_{1,c, \Omega_j} \lVert w\rVert_{1,c, \Omega_j},
\end{equation}
where $C_\textup{dPU}$ appears in \eqref{eq:2-21}.
Moreover, {for all $v_h, w_h \in \mathcal{V}^h_j$
\begin{equation}
\label{eq:RCD-cDF}
\lvert (v_h,\Pi^h(\chi_j w_h))_{1,c,\Omega_j} - (\Pi^h(\chi_j v_h),w_h)_{1,c,\Omega_j}  \rvert  \le C_{DF,j} \lVert v_h \rVert_{1,c, \Omega_j} \lVert w_h\rVert_{1,c, \Omega_j}
\end{equation}
which is the finite element expression of assumption~\eqref{eq:cDF}, where 
\begin{equation}
\label{eq:RCD-cDFconst}
C_{DF,j} =  C_\textup{dPU} \frac{\nu_{+,j}}{\sqrt{\tilde{c}_{-,j}\nu_{-,j}}} \frac{1}{\delta} + 2 C_{\textup{err},j},
\end{equation}
} 
with $ C_{\textup{err},j}$ given by \eqref{eq:Cerrj}.
\end{lemma}
\begin{proof}
Note that
\[
\begin{split}
&(v,\chi_j w)_{1,c,\Omega_j} - (\chi_j v,w)_{1,c,\Omega_j} \\
&= \int_{\Omega_j} \nu \nabla v \cdot (w \nabla \chi_j + \chi_j \nabla w) - \int_{\Omega_j} \nu (v \nabla \chi_j + \chi_j \nabla v) \cdot \nabla w \\
&= \int_{\Omega_j} \nu \nabla\chi_j \cdot (w \nabla v - v \nabla w). 
\end{split}
\]
Then, by the Cauchy-Schwarz inequality and \eqref{eq:2-21}
\[
\begin{split}
& \lvert (v,\chi_j w)_{1,c,\Omega_j} - (\chi_j v,w)_{1,c,\Omega_j}  \rvert  \\  
& \le \nu_{+,j} \, C_\textup{dPU} \frac{1}{\delta} \bigl( \lVert w \rVert_{\mathrm{L}^2(\Omega_j)}  \lVert \nabla v \rVert_{\mathrm{L}^2(\Omega_j)}
+ \lVert v \rVert_{\mathrm{L}^2(\Omega_j)} \lVert \nabla w \rVert_{\mathrm{L}^2(\Omega_j)} \bigr)\\
& =  \frac{C_\textup{dPU}}{\delta} \frac{\nu_{+,j}}{\sqrt{\tilde{c}_{-,j}\nu_{-,j}}} \Bigl( \sqrt{\tilde{c}_{-,j}} \lVert w \rVert_{\mathrm{L}^2(\Omega_j)} \sqrt{\nu_{-,j}} \lVert \nabla v \rVert_{\mathrm{L}^2(\Omega_j)} \\
&\qquad\qquad\qquad\qquad\qquad + \sqrt{\tilde{c}_{-,j}} \lVert v \rVert_{\mathrm{L}^2(\Omega_j)} \sqrt{\nu_{-,j}} \lVert \nabla w \rVert_{\mathrm{L}^2(\Omega_j)} \Bigr)\\
& = \frac{C_\textup{dPU}}{\delta} \frac{\nu_{+,j}}{\sqrt{\tilde{c}_{-,j}\nu_{-,j}}} 
\begin{pmatrix} \sqrt{\tilde{c}_{-,j}} \lVert w \rVert_{\mathrm{L}^2(\Omega_j)} & \sqrt{\nu_{-,j}} \lVert \nabla w \rVert_{\mathrm{L}^2(\Omega_j)} \end{pmatrix}
\begin{pmatrix} \sqrt{\nu_{-,j}} \lVert \nabla v \rVert_{\mathrm{L}^2(\Omega_j)} \\ \sqrt{\tilde{c}_{-,j}} \lVert v \rVert_{\mathrm{L}^2(\Omega_j)} \end{pmatrix} \\
& \le  \frac{C_\textup{dPU}}{\delta} \frac{\nu_{+,j}}{\sqrt{\tilde{c}_{-,j}\nu_{-,j}}}
\Bigl( \tilde{c}_{-,j} \lVert w \rVert_{\mathrm{L}^2(\Omega_j)}^2 + \nu_{-,j} \lVert \nabla w \rVert_{\mathrm{L}^2(\Omega_j)}^2 \Bigr)^{1/2}  \\ 
&\qquad\qquad\qquad\qquad\qquad \cdot\Bigl( \tilde{c}_{-,j} \lVert v \rVert_{\mathrm{L}^2(\Omega_j)}^2 + \nu_{-,j} \lVert \nabla v \rVert_{\mathrm{L}^2(\Omega_j)}^2 \Bigr)^{1/2} \\
& \le \frac{C_\textup{dPU}}{\delta} \frac{\nu_{+,j}}{\sqrt{\tilde{c}_{-,j}\nu_{-,j}}}  \lVert v \rVert_{1,c, \Omega_j} \lVert w\rVert_{1,c, \Omega_j},
\end{split}
\]
where at the end we have used the Cauchy-Schwarz inequality with respect to the Euclidean inner product in $\mathbb{R}^2$. 

For $C_{DF,j}$ we find the continuous analogue of the left-hand side in \eqref{eq:cDF}: 
for $\mathbf{V}^j, \mathbf{W}^j \in \mathbb{R}^{n_j}$ vectors of degrees of freedom for local functions $v_h, w_h \in \mathcal{V}^h_j$
\[
\begin{split}
&\lvert ([D_j, F_{\Omega_j}] \mathbf{V}^j, \mathbf{W}^j) \rvert 
= \lvert (F_{\Omega_j} \mathbf{V}^j, D_j\mathbf{W}^j) - (F_{\Omega_j} D_j \mathbf{V}^j, \mathbf{W}^j)\rvert \\
& = \lvert (v_h,\Pi^h(\chi_j w_h))_{1,c,\Omega_j} - (\Pi^h(\chi_j v_h),w_h)_{1,c,\Omega_j}  \rvert \\
& = \lvert ((I-\Pi^h)(\chi_j v_h),w_h)_{1,c,\Omega_j} - (v_h,(I-\Pi^h)(\chi_j w_h))_{1,c,\Omega_j} \\
& \qquad + (v_h,\chi_j w_h)_{1,c,\Omega_j} - (\chi_j v_h,w_h)_{1,c,\Omega_j}  \rvert.
\end{split}
\]
Now, by the Cauchy-Schwarz inequality and \eqref{eq:RCDsim3-10bis}
\[
\lvert ((I-\Pi^h)(\chi_j v_h),w_h)_{1,c,\Omega_j} \rvert  \le C_{\textup{err},j}  \lVert v_h \rVert_{1,c, \Omega_j} \lVert w_h \rVert_{1,c, \Omega_j}
\]
and similarly for $\lvert (v_h,(I-\Pi^h)(\chi_j w_h))_{1,c,\Omega_j} \rvert$, so, combining with \eqref{eq:RCDsimLemma3-7i}, we get 
\[
\begin{split}
\lvert ([D_j, F_{\Omega_j}] \mathbf{V}^j, \mathbf{W}^j) \rvert  & \le 
\left ( C_\textup{dPU} \frac{\nu_{+,j}}{\sqrt{\tilde{c}_{-,j}\nu_{-,j}}} \frac{1}{\delta} + 2 C_{\textup{err},j} \right )   \lVert v_h \rVert_{1,c, \Omega_j} \lVert w_h\rVert_{1,c, \Omega_j} \\
& = \left ( C_\textup{dPU} \frac{\nu_{+,j}}{\sqrt{\tilde{c}_{-,j}\nu_{-,j}}} \frac{1}{\delta} + 2 C_{\textup{err},j} \right ) \lVert \mathbf{V}^j \rVert_{\Omega_j}  \lVert \mathbf{W}^j \rVert_{\Omega_j}.
\end{split}
\]
\end{proof}

Finally, for assumption \eqref{eq:cDB} let us study the commutator between the partition of unity matrix and the local problem matrix. 
\begin{lemma}[{$C_{DB,j}$ in \eqref{eq:cDB}}]
For all $v,w \in \mathrm{H}^1(\Omega_j)$ 
\begin{equation}
\label{eq:RCDsimLemma3-7iDB}
\lvert  a_j(v,\chi_j w)- a_j(\chi_j v,w) \rvert \le  
C_\textup{dPU} \left ( \frac{\nu_{+,j}}{\sqrt{\tilde{c}_{-,j}\nu_{-,j}}}+ \frac{\lVert \mathbf{a} \rVert_{\mathrm{L}^\infty(\Omega_j)}}{\tilde{c}_{-,j}} \right)  \frac{1}{\delta}   \lVert v \rVert_{1,c, \Omega_j} \lVert w\rVert_{1,c, \Omega_j}
\end{equation}
where $C_\textup{dPU}$ appears in \eqref{eq:2-21}.
Moreover, {for all $v_h, w_h \in \mathcal{V}^h_j$
\begin{equation}
\label{eq:RCD-cDB}
\lvert a_j(v_h,\Pi^h(\chi_j w_h)) - a_j(\Pi^h(\chi_j v_h),w_h)  \rvert  \le C_{DB,j} \lVert v_h \rVert_{1,c, \Omega_j} \lVert w_h\rVert_{1,c, \Omega_j}
\end{equation}
which is the finite element expression of assumption~\eqref{eq:cDB}, where} 
\begin{equation}
\label{eq:RCD-cDBconst}
C_{DB,j} =  C_\textup{dPU} \left ( \frac{\nu_{+,j}}{\sqrt{\tilde{c}_{-,j}\nu_{-,j}}} + \frac{\lVert \mathbf{a} \rVert_{\mathrm{L}^\infty(\Omega_j)}}{\tilde{c}_{-,j}} \right)  \frac{1}{\delta} + 2 C_{\textup{cont},j}  C_{\textup{err},j},  
\end{equation}
with $C_{\textup{cont},j}$, $C_{\textup{err},j}$ given by \eqref{eq:Ccontj}, \eqref{eq:Cerrj}.
\end{lemma}
\begin{proof}
Note that
\[
\begin{split}
&a_j(v,\chi_j w)- a_j(\chi_j v,w) = \frac{1}{2} \int_{\Omega_j} \chi_j w \mathbf{a} \cdot \nabla v - v \mathbf{a} \cdot (w \nabla \chi_j + \chi_j \nabla w) + \\
&-  \frac{1}{2} \int_{\Omega_j} w \mathbf{a} (v \nabla \chi_j + \chi_j \nabla v) - \chi_j v \mathbf{a} \cdot \nabla w\\
& + \int_{\Omega_j} \nu \nabla v \cdot (w \nabla \chi_j + \chi_j \nabla w) - \int_{\Omega_j} \nu (v \nabla \chi_j + \chi_j \nabla v) \cdot \nabla w \\
&= - \int_{\Omega_j} v w \, \mathbf{a} \cdot \nabla \chi_j + \int_{\Omega_j} \nu \nabla\chi_j \cdot (w \nabla v - v \nabla w). 
\end{split}
\]
By the Cauchy-Schwarz inequality and \eqref{eq:2-21}
\[
\begin{split}
\left \lvert \int_{\Omega_j} v w \, \mathbf{a} \cdot \nabla \chi_j \right \rvert 
& \le C_\textup{dPU} \frac{1}{\delta} \lVert \mathbf{a} \rVert_{\mathrm{L}^\infty(\Omega_j)}  \lVert v \rVert_{\mathrm{L}^2(\Omega_j)} \lVert w \rVert_{\mathrm{L}^2(\Omega_j)} \\
& \le C_\textup{dPU} \frac{1}{\delta} \frac{\lVert \mathbf{a} \rVert_{\mathrm{L}^\infty(\Omega_j)}}{\tilde{c}_{-,j}} \lVert v \rVert_{1,c, \Omega_j} \lVert w\rVert_{1,c, \Omega_j}.
\end{split}
\]
Therefore, proceeding for the other term as in Lemma \ref{lem:RCDsimLemma3-7i}, 
\[
\lvert  a_j(v,\chi_j w)- a_j(\chi_j v,w) \rvert \le 
C_\textup{dPU} \Biggl ( \frac{\nu_{+,j}}{\sqrt{\tilde{c}_{-,j}\nu_{-,j}}} + \frac{\lVert \mathbf{a} \rVert_{\mathrm{L}^\infty(\Omega_j)}}{\tilde{c}_{-,j}} \Biggr)  \frac{1}{\delta} 
\lVert v \rVert_{1,c, \Omega_j} \lVert w\rVert_{1,c, \Omega_j}.
\]

For $C_{DB,j}$ we find the continuous analogue of the left-hand side in \eqref{eq:cDB}: 
for $\mathbf{V}^j, \mathbf{W}^j \in \mathbb{R}^{n_j}$ vectors of degrees of freedom for local functions $v_h, w_h \in \mathcal{V}^h_j$
\[
\begin{split}
&\lvert ([D_j, B_j] \mathbf{V}^j, \mathbf{W}^j) \rvert 
= \lvert (B_j \mathbf{V}^j, D_j\mathbf{W}^j) - (B_j D_j \mathbf{V}^j, \mathbf{W}^j)\rvert \\
& = \lvert a_j(v_h,\Pi^h(\chi_j w_h)) - a_j(\Pi^h(\chi_j v_h),w_h)  \rvert \\
& = \lvert a_j((I-\Pi^h)(\chi_j v_h),w_h) - a_j(v_h,(I-\Pi^h)(\chi_j w_h)) \\
& \qquad + a_j(v_h,\chi_j w_h) - a_j(\chi_j v_h,w_h) \rvert.
\end{split}
\]
Now, by the continuity property \eqref{eq:RCDloccont} of $a_j$ and \eqref{eq:RCDsim3-10bis}
\[
\lvert a_j((I-\Pi^h)(\chi_j v_h),w_h)  \rvert 
\le C_{\textup{cont},j}  C_{\textup{err},j}  \lVert v_h \rVert_{1,c, \Omega_j} \lVert w_h \rVert_{1,c, \Omega_j}
\]
and similarly for $\lvert a_j(v_h,(I-\Pi^h)(\chi_j w_h)) \rvert$, so, combining with \eqref{eq:RCDsimLemma3-7iDB}, we get 
\[
\begin{split}
&\lvert ([D_j, B_j] \mathbf{V}^j, \mathbf{W}^j) \rvert  \\
&\le \left [ C_\textup{dPU} \left ( \frac{\nu_{+,j}}{\sqrt{\tilde{c}_{-,j}\nu_{-,j}}} + \frac{\lVert \mathbf{a} \rVert_{\mathrm{L}^\infty(\Omega_j)}}{\tilde{c}_{-,j}} \right)  \frac{1}{\delta} + 2 C_{\textup{cont},j}  C_{\textup{err},j}  \right ]
\lVert v_h \rVert_{1,c, \Omega_j} \lVert w_h\rVert_{1,c, \Omega_j} \\
& = \left [ C_\textup{dPU} \left ( \frac{\nu_{+,j}}{\sqrt{\tilde{c}_{-,j}\nu_{-,j}}} + \frac{\lVert \mathbf{a} \rVert_{\mathrm{L}^\infty(\Omega_j)}}{\tilde{c}_{-,j}} \right)  \frac{1}{\delta} + 2 C_{\textup{cont},j}  C_{\textup{err},j}  \right ]
 \lVert \mathbf{V}^j \rVert_{\Omega_j}  \lVert \mathbf{W}^j \rVert_{\Omega_j}.
\end{split}
\]
\end{proof}

\subsection{Summary of the constants}
For the heterogeneous reaction-convection-diffusion problem \eqref{eq:RCDbvp} we have proved that the upper and lower bounds of Theorem~\ref{thm:main} 

\begin{equation*}
\max_{\mathbf{V} \in \mathbb{R}^n} \; \frac{\Vert M^{-1} A \mathbf{V} \rVert_\Omega}{\Vert \mathbf{V} \rVert_\Omega} 
\le \sqrt{\Lambda_0 \Lambda_1} \max_{j=1,\dots,N} \{ C_{D,j} (C_{\textup{stab},j} C_{DB,j} + C_{D,j})\}
\end{equation*}
\begin{multline*}
\min_{\mathbf{V} \in \mathbb{R}^n} \frac{\lvert (F_\Omega \mathbf{V}, M^{-1}A \mathbf{V}) \rvert}{ \lVert \mathbf{V} \rVert_\Omega^2} \ge 
 \frac{1}{\Lambda_0} 
-  \Lambda_1 \max_{j=1,\dots,N} \{ C_{D,j} C_{\textup{stab},j}C_{DB,j} \} \\
- \Lambda_1 \max_{j=1,\dots,N} \{ C_{DF,j} (C_{\textup{stab},j}C_{DB,j} +C_{D,j}) \}
\end{multline*}

\noindent
hold with the constants {$\Lambda_0$, $\Lambda_1$ from Lemma~\ref{lem:RCDconverseStSp}, Lemma~\ref{lem:RCDfinOvr}:}
\[
\begin{split}
& \Lambda_0 = \max_{j=1,\dots,N} \, \# \Lambda(j), \quad \text{where } \Lambda(j) = \set{j' | \Omega_j \cap \Omega_{j'} \ne \emptyset}\\
& \Lambda_1 = \max \set{m | \exists \, j_1 \ne \dots \ne j_m \text{ such that } \textup{meas}(\Omega_{j_1} \cap \dots \cap \Omega_{j_m}) \ne 0}
\end{split}
\]


\noindent
$C_{\textup{stab},j}$ from Lemma~\ref{lem:RCDstab}: 
\[
C_{\textup{stab},j} = 1
\]

\noindent
$C_{D,j}$ from \eqref{eq:RCDsim3-16const}:
\[
C_{D,j} =  \sqrt{2} \left ( 1 + C_\textup{dPU} \sqrt{\frac{\nu_{+,j}}{\tilde{c}_{-,j}}} \frac{1}{\delta} \right )  + C_{\textup{err},j}
\]

\noindent
$C_{DF,j}$ from \eqref{eq:RCD-cDFconst}:
\[
C_{DF,j} = C_\textup{dPU} \frac{\nu_{+,j}}{\sqrt{\tilde{c}_{-,j}\nu_{-,j}}} \frac{1}{\delta} + 2 C_{\textup{err},j}
\]

\noindent
$C_{DB,j}$ from \eqref{eq:RCD-cDBconst}:
\[
C_{DB,j} = C_\textup{dPU} \left ( \frac{\nu_{+,j}}{\sqrt{\tilde{c}_{-,j}\nu_{-,j}}} + \frac{\lVert \mathbf{a} \rVert_{\mathrm{L}^\infty(\Omega_j)}}{\tilde{c}_{-,j}} \right)  \frac{1}{\delta} + 2 C_{\textup{cont},j}  C_{\textup{err},j}  
\]

\noindent where from  \eqref{eq:Ccontj}
\[
C_{\textup{cont},j} = \frac{\tilde{c}_{+,j}}{\tilde{c}_{-,j}} \frac{\nu_{+,j}}{\nu_{-,j}} +
\frac{1}{2}  \frac{\lVert \mathbf{a} \rVert_{\mathrm{L}^\infty(\Omega_j)}}{\sqrt{\tilde{c}_{-,j}\nu_{-,j} }} +
\frac{\lVert \alpha \rVert_{\mathrm{L}^\infty(\Omega)} C_\textup{tr}}{\sqrt{\tilde{c}_{-,j}}} 
\left( \frac{1}{H_\textup{sub} \sqrt{\tilde{c}_{-,j}}} + \frac{1}{2\sqrt{\nu_{-,j}}} \right)
\]
and from \eqref{eq:Cerrj}
\[
C_{\textup{err},j} = C_\Pi \, c(r,d) \, C_\textup{dPU} \, \sqrt{C_\textup{inv}} \left (\sqrt{\frac{\nu_{+,j}}{\nu_{-,j}}} + \sqrt{\frac{\tilde{c}_{+,j}}{\nu_{-,j}}} h \right) \frac{h}{\delta},
\]
and $C_\textup{tr}$ appears in Lemma~\ref{lem:MTI}, $C_\Pi$ in \eqref{eq:errEst}, $C_\textup{dPU}$ in \eqref{eq:2-21}, and $C_\textup{inv}$ is a standard inverse inequality constant (see the proof of Lemma \ref{lem:errInterpChi} for more details), and $c(r,d) = \max_{\vert \gamma\vert = r}\sum_{\beta \,|\, 0 < \beta\le\gamma}\binom{\gamma}{\beta}$.

These estimates can be then specialized for particular regimes of the physical coefficients of the equation or of the numerical parameters. 
Note that the lower bound is interesting only if the positive term dominates the negative ones in the considered regime. 
In particular, if the overlap $\delta$ is sufficiently generous, both negative terms can be made arbitrarily small. So we have proved for the SORAS algorithm that a larger overlap helps the convergence of the domain decomposition preconditioner, as expected.   

For instance, if the equation in \eqref{eq:RCDbvp} derives from a backward Euler scheme for solving the time-dependent convection-diffusion problem, we would have $\tilde{c} = 1/\Delta t$, where $\Delta t$ is the time step of the scheme. Now, note that the constants $C_{D,j}, C_{DB,j}, C_{DF,j}$ appearing in the negative terms contain the adimensional quantities 
\[
\sqrt{\frac{\, \nu \,}{\, \tilde{c} \, }} \frac{1}{\delta}, \qquad    \frac{\lVert \mathbf{a} \rVert_{\mathrm{L}^\infty(\Omega_j)}}{\tilde{c}} \frac{1}{\delta}, 
\]
(where we have considered the homogeneous case for simplicity). 
Hence for these quantities to be small, the overlap $\delta$ should be asymptotically bigger than the square root of the diffusion area covered in a time step, and than the convection distance covered in a time step. 
Therefore, on the one hand when the diffusion coefficient or the convection velocity grow, the overlap size should be increased; on the other hand if the time discretization step shrinks, one could take a smaller overlap. 
Furthermore, the interpolation constant $C_{\textup{err},j}$, also appearing in $C_{D,j}, C_{DB,j}, C_{DF,j}$, leads to restrictions involving the mesh size $h$ and the overlap $\delta$. 

The lower bound on the field of values could be improved by designing a suitable coarse space to add a second level to the standard SORAS preconditioner.
Note that for generic symmetric positive definite problems, robust lower bounds \emph{on the spectrum} can be indeed obtained in this manner \cite{HaJoNa:soras:2015}, but 
for generic non-self-adjoint or indefinite problems this currently constitutes a major challenge. 

\subsection{Numerical experiments}

To conclude, we test numerically the performance of the preconditioner on the reaction-convection-diffusion problem~\eqref{eq:RCDbvp} with $\Omega$ a rectangle $[0,N\cdot0.2]\times[0,0.2]$ (where $N$ is the number of subdomains), $\Gamma_D = \Gamma$, $\Gamma_R=\emptyset$; for the local problems of the preconditioner, Robin transmission conditions with parameter $\alpha$ as in \eqref{eq:exalpha} are imposed on the subdomain interfaces. 
In Tables~\ref{tab:rotating},\ref{tab:negdiv},\ref{tab:normal} we take $N=5$ and $f=100\exp{\{-10((x-0.5)^2+(y-0.1)^2)\}}$, which is centered at the barycenter of $\Omega$. In Tables~\ref{tab:wscalingReg}, \ref{tab:wscalingMetis} we vary $N$ to test weak scaling, where the size of the problem increases about proportionally to $N$ (note that the width of $\Omega$ above is proportional to $N$), and we take $f=100\exp{\{-10((x-0.1)^2+(y-0.1)^2)\}}$, which is centered on the left of $\Omega$. 
The problem is discretized by piece-wise linear Lagrange finite elements on a uniform triangular mesh with $60$ points on the vertical side of the rectangle and $N \cdot 60$ points on the horizontal one, resulting in 
$18361$ degrees of freedom for $N=5$, and $7381$, $14701$, $29341$, $58621$, $117181$, $234301$ degrees of freedom for $N=2,4,8,16,32,64$ respectively.  
In Tables~\ref{tab:rotating}--\ref{tab:wscalingReg} the domain is partitioned into $N$ vertical strips, while in Table~\ref{tab:wscalingMetis} we consider arbitrary partitions into $N$ irregular subdomains obtained using the automatic mesh partitioner METIS \cite{KaKu:1998:metis}; then each subdomain is augmented with mesh elements layers of size $\delta/2$ to obtain the overlapping decomposition (the total width of the overlap between two subdomains is then $\delta$). 

GMRES with right preconditioning is stopped when the relative residual is reduced by $10^{-6}$. In Tables~\ref{tab:rotating},\ref{tab:negdiv},\ref{tab:normal}  we take a zero initial guess, while in Tables~\ref{tab:wscalingReg}, \ref{tab:wscalingMetis} we take a random initial guess. We test SORAS preconditioner \eqref{eq:SORAS} and also ORAS preconditioner: 
\[
M^{-1}_{ORAS} \coloneqq \sum_{j=1}^N R_j^T D_j B_j^{-1} R_j. 
\]
In the tables we use \# to denote the number of iterations for convergence. To apply the preconditioner, the local problems in each subdomain are solved with the direct solver  MUMPS \cite{amestoy:2001:fully}.  
All the computations are done in the \href{https://doc.freefem.org/documentation/ffddm/index.html}{ffddm} framework of \href{https://freefem.org/}{FreeFEM}, an open source domain specific language (DSL) specialised for solving boundary value problems with variational methods. 

\begin{table}{
\begin{center}
\begin{tabular}{|l|c|c|c|c|}
\hline
 & \multicolumn{4}{c|}{\#SORAS(ORAS)}\tabularnewline
\cline{2-5} $\mathbf{a}= 2\pi[-(y-0.1), (x-0.5)]^T$ & $\delta=2h$ & $\delta=4h$ & $\delta=6h$ & $\delta=8h$ \tabularnewline
 \hline
 $c_0=1, \; \nu=1$        & 21(18) & 20(14) & 20(12) & 19(11) \tabularnewline
 \hline
 $c_0=1, \; \nu=0.001$ & 14(9) & 13(6) & 12(5) & 12(5) \tabularnewline
 \hline
 $c_0=0.001, \; \nu=1$        & 21(20) & 20(15) & 20(13) & 19(11) \tabularnewline
 \hline
 $c_0=0.001, \; \nu=0.001$ & 15(10) & 14(7) & 13(5) & 13(5) \tabularnewline
 \hline
\end{tabular}
\caption{Iteration numbers for SORAS(ORAS) preconditioners in the case of a convection field $\mathbf{a}= 2\pi[-(y-0.1), (x-0.5)]^T$, for different values of the overlap $\delta$, the reaction coefficient $c_0$ and the viscosity $\nu$. 
The domain is decomposed into $N=5$ overlapping vertical strips and the global problem has 18361 degrees of freedom.} 
\label{tab:rotating} 
\end{center}
}
\end{table}
We examine several configurations for the coefficients in \eqref{eq:RCDbvp}. First, in Table~\ref{tab:rotating} we consider a rotating convection field $\mathbf{a}= 2\pi[-(y-0.1), (x-0.5)]^T$ and small/large values for the reaction coefficient $c_0$ and the viscosity $\nu$. We can see that a larger overlap helps the convergence of the  preconditioners, as expected. The number of iterations appears not very sensitive to $c_0$, while it increases when $\nu$ is larger. 
ORAS preconditioner performs better than SORAS preconditioner, but currently the convergence of ORAS preconditioner can not be rigorously analyzed. 

\begin{table}{
\begin{center}
\begin{tabular}{|l|c|c|c|c|}
\hline
 & \multicolumn{4}{c|}{\#SORAS(ORAS)}\tabularnewline
\cline{2-5} $\mathbf{a}= [-x, -y]^T$ & $\delta=2h$ & $\delta=4h$ & $\delta=6h$ & $\delta=8h$ \tabularnewline
 \hline
 $c_0=1, \; \nu=1$        & 21(19) & 21(14) & 20(13) & 20(11) \tabularnewline
 \hline
 $c_0=1, \; \nu=0.001$ & 16(7) & 16(7) & 16(6) & 16(6) \tabularnewline
 \hline
 $c_0=0.001, \; \nu=1$        & 22(24) & 22(18) & 22(15) & 21(13) \tabularnewline
 \hline
 $c_0=0.001, \; \nu=0.001$ & 17(8) & 16(7) & 16(7) & 16(6) \tabularnewline
 \hline
\end{tabular}
\caption{Repeat of Table~\ref{tab:rotating} but with $\mathbf{a}= [-x, -y]^T$. In this case $\dive \mathbf{a}=-2$ is negative and $\tilde{c}=c_0-1$ does not verify condition~\eqref{eq:hypctilde}.} 
\label{tab:negdiv} 
\end{center}
}
\end{table}
Then, in Table~\ref{tab:negdiv} we take $\mathbf{a}= [-x, -y]^T$, which has negative divergence $\dive \mathbf{a}=-2$, to test the robustness of the method when condition~\eqref{eq:hypctilde} on the positiveness of $\tilde{c}$ is violated: in this case, $\tilde{c}=c_0-1$, so $\tilde{c}=0$,  $\tilde{c}=-0.999$ for $c_0=1$, $c_0=0.001$ respectively. We can observe that both preconditioners still perform well. 

\begin{table}{
\begin{center}
\begin{tabular}{|l|c|c|c|c|}
\hline
 & \multicolumn{4}{c|}{\#SORAS(ORAS)}\tabularnewline
\cline{2-5} $\mathbf{a}=[1, 0]^T$ & $\delta=2h$ & $\delta=4h$ & $\delta=6h$ & $\delta=8h$ \tabularnewline
 \hline
 $c_0=1, \; \nu=1$        & 20(18) & 20(15) & 20(13) & 20(12) \tabularnewline
 \hline
 $c_0=1, \; \nu=0.001$ & 11(6) & 11(5) & 11(5) & 11(5) \tabularnewline
 \hline
 $c_0=0.001, \; \nu=1$        & 20(20) & 20(16) & 20(14) & 20(13) \tabularnewline
 \hline
 $c_0=0.001, \; \nu=0.001$ & 12(6) & 12(5) & 12(5) & 12(5) \tabularnewline
 \hline
\end{tabular}
\caption{Repeat of Table~\ref{tab:rotating} but with $\mathbf{a}=[1, 0]^T$ and with Streamline Upwind Petrov-Galerkin stabilization for the Galerkin approximation.} 
\label{tab:normal} 
\end{center}
}
\end{table}
Finally, in Table~\ref{tab:normal} we consider a horizontal convecting field $\mathbf{a}=[1, 0]^T$, which is normal to the interfaces between subdomains. In this case non-physical numerical instabilities appear in the solution. Note that this is a discretization issue, not related to the preconditioner: the robust direct solver MUMPS yields the same instabilities in the numerical solution as the preconditioned GMRES solver. 
To stabilize the discrete variational formulation, we use the Streamline Upwind Petrov-Galerkin (SUPG) method, which adds to the Galerkin approximation the following term (see for instance \cite[§11.8.6]{quarteroni:2009:book}): 
\[
\mathcal{L}_h(u_h,f;v_h) = \theta \sum_{\tau\in \mathcal{T}^h} \int_\tau (\mathcal{L}u_h-f) \, \frac{h_\tau}{\rvert \mathbf{a}\rvert} \, \mathcal{L}_{SS}v_h, 
\]
where $\theta$ is a stabilization parameter (here we choose $\theta=0.15$), $h_\tau$ is the diameter of the mesh element $\tau$, and 
\[
\mathcal{L}u_h = c_0 u_h + \dive(\mathbf{a}u_h) -\dive(\nu \nabla u_h), \quad
\mathcal{L}_{SS}v_h = \frac{1}{2}\dive(\mathbf{a}v_h)+\frac{1}{2}\mathbf{a}\cdot\nabla v_h.
\]
We can see that ORAS preconditioner performs better than SORAS preconditioner, but depends more on the overlap size. 

\begin{table}{
\begin{center}
\begin{tabular}{|l|c|c|c|c|c|c|}
\hline
 & \multicolumn{6}{c|}{\#SORAS(ORAS)}\tabularnewline
\cline{2-7} $\mathbf{a}=[1, 0]^T$ & $N=2$ & $N=4$ & $N=8$ & $N=16$ & $N=32$ & $N=64$ \tabularnewline
 \hline
 $c_0=1, \; \nu=1$        & 18(15) & 23(18) & 28(19) & 35(19) & 36(19) & 36(19) \tabularnewline
 \hline
 $c_0=1, \; \nu=0.001$ & 8(3) & 10(5) & 16(8) & 23(16) & 37(32) & 63(62) \tabularnewline
 \hline
 $c_0=0.001, \; \nu=1$        & 18(15) & 23(19) & 29(21) & 35(21) & 36(21) & 36(21) \tabularnewline
 \hline
 $c_0=0.001, \; \nu=0.001$ & 8(3) & 10(5) & 16(8) & 24(16) & 40(32) & 71(64) \tabularnewline
 \hline
\end{tabular}
\caption{Weak scaling test, with a regular decomposition into $N$ vertical strips ($\delta=4h$). The global problem has $7381$, $14701$, $29341$, $58621$, $117181$, $234301$ degrees of freedom for $N=2,4,8,16,32,64$ subdomains respectively.} 
\label{tab:wscalingReg} 
\end{center}
}
\end{table}

\begin{table}{
\begin{center}
\begin{tabular}{|l|c|c|c|c|c|c|}
\hline
 & \multicolumn{6}{c|}{\#SORAS(ORAS)}\tabularnewline
\cline{2-7} $\mathbf{a}=[1, 0]^T$ & $N=2$ & $N=4$ & $N=8$ & $N=16$ & $N=32$ & $N=64$ \tabularnewline
 \hline
 $c_0=1, \; \nu=1$        & 21(17) & 30(22) & 40(23) & 48(23) & 53(23) & 55(23) \tabularnewline
 \hline
 $c_0=1, \; \nu=0.001$ & 10(4) & 12(5) & 17(9) & 25(17) & 38(32) & 63(63) \tabularnewline
 \hline
 $c_0=0.001, \; \nu=1$        & 21(18) & 30(25) & 40(28) & 48(27) & 54(28) & 57(29) \tabularnewline
 \hline
 $c_0=0.001, \; \nu=0.001$ & 10(4) & 12(5) & 18(9) & 26(17) & 42(33) & 73(65) \tabularnewline
 \hline
\end{tabular}
\caption{Weak scaling test as in Table~\ref{tab:wscalingReg}, but with METIS decomposition into $N$ arbitrary-shaped subdomains.} 
\label{tab:wscalingMetis} 
\end{center}
}
\end{table}
Now, again in this third configuration with $\mathbf{a}=[1, 0]^T$ and SUPG stabilization, we perform a weak scaling test by taking $\Omega = [0,N\cdot0.2]\times[0,0.2]$ for increasing number of subdomains $N$. First we consider a regular decomposition into vertical strips (Table~\ref{tab:wscalingReg}) and then an arbitrary decomposition made by METIS (Table~\ref{tab:wscalingMetis}). We fix the overlap $\delta=4h$. 
Comparing Table~\ref{tab:wscalingReg} with Table~\ref{tab:wscalingMetis}, we can see that the number of iterations is higher when taking arbitrary-shaped subdomains. Moreover, in the cases with $\nu=0.001$, convergence deteriorates with $N$, which shows the need of designing robust two-level preconditioners.



\color{black}


%
%

\bibliographystyle{spmpsci}      
\bibliography{paper_BCNT}   

%
%

\end{document}